\documentclass[11pt,reqno]{amsart}
\usepackage{amsmath}
\usepackage{a4wide}
\usepackage[utf8]{inputenc}
\usepackage{amssymb}
\usepackage{amsopn}
\usepackage{epsfig}
\usepackage{amsfonts}
\usepackage{latexsym}
\usepackage{amsthm}
\usepackage{enumerate}
\usepackage[UKenglish]{babel}
\usepackage{verbatim}
\usepackage{color}
\usepackage{calrsfs}
\usepackage{pgf}
\usepackage{tikz}
\usepackage{breqn}
\usepackage{url}
\usepackage{cleveref}
\usepackage{mathrsfs}

\usepackage[leqno]{amsmath}

\makeatletter
\newcommand{\leqnomode}{\tagsleft@true\let\veqno\@@leqno}
\newcommand{\reqnomode}{\tagsleft@false\let\veqno\@@eqno}
\makeatother
\crefformat{section}{\S#2#1#3}
\usetikzlibrary{arrows,automata}
\DeclareMathAlphabet{\pazocal}{OMS}{zplm}{m}{n}

\newtheorem{theorem}{Theorem}[section]
\newtheorem{lemma}[theorem]{Lemma}
\newtheorem{proposition}[theorem]{Proposition}
\newtheorem{corollary}[theorem]{Corollary}
\newtheorem{main}{Theorem}

\theoremstyle{definition}
\newtheorem{definition}[theorem]{Definition}

\theoremstyle{remark}

\numberwithin{equation}{section}
\newtheorem{problem}{Problem}

\newcommand{\Na}{\mathbb{N}}
\newcommand{\N}{\ensuremath{\mathbb{N}}} % naturales
\newcommand{\R}{\ensuremath{\mathbb{R}}} % reales
\newcommand{\Z}{\ensuremath{\mathbb{Z}}} % enteros
\newcommand{\Q}{\ensuremath{\mathbb{Q}}} %racionales
\newcommand{\set}[1]{\left\{#1\right\}} %\conjuntos

\newcommand{\f}{\infty}

\newcommand{\si}{\sigma}

\newcommand{\veps}{\varepsilon}
\newcommand{\vphi}{\varphi}
\newcommand{\Lu}{\mathcal{L}} %Transformacion de Luroth
\newcommand{\Ga}{\mathcal{G}} %Transformacion de Gauss
\newcommand{\dimh}{\dim_{\operatorname{H}}} %dimension de Hausdorff
 
\newcommand{\noD}{\operatorname{D}}%distal
\newcommand{\AS}{\operatorname{AS}}%asymtotic
\newcommand{\LY}{\operatorname{LY}}%li-yorke
\newcommand{\PR}{\operatorname{PR}}%proximal

\DeclareMathOperator{\pref}{pref}
\newcommand{\bfA}{\mathbf{A}}
\newcommand{\clA}{\mathcal{A}}
\newcommand{\clD}{\mathcal{D}}
\newcommand{\clF}{\mathcal{F}}

\newcommand{\clI}{\mathcal{I}}
\newcommand{\scA}{\mathscr{A}}
\newcommand{\scB}{\mathscr{B}}
\newcommand{\scD}{\mathscr{D}}
\newcommand{\scG}{\mathscr{G}}
\newcommand{\scI}{\mathscr{I}}
\newcommand{\scL}{\mathscr{L}}
\newcommand{\scR}{\mathscr{R}}

\newcommand{\scT}{\mathscr{T}}
\newcommand{\scY}{\mathscr{Y}}
\newcommand{\leb}{\mathfrak{m}}%medida de Lebesgue
\newcommand{\bfa}{\mathbf{a}}
\newcommand{\bfb}{\mathbf{b}}
\newcommand{\bfc}{\mathbf{c}}
\newcommand{\bfd}{\mathbf{d}}
\newcommand{\bfe}{\mathbf{e}}
\newcommand{\bff}{\mathbf{f}}
\newcommand{\bfg}{\mathbf{g}}
\newcommand{\bfr}{\mathbf{r}}
\newcommand{\bfx}{\mathbf{x}}
\newcommand{\bfy}{\mathbf{y}}
\newcommand{\sanu}{(a_n)_{n\geq 1}}
\newcommand{\sabu}{(b_n)_{n\geq 1}}
\newcommand{\saxu}{(x_n)_{n\geq 1}}
\newcommand{\sayu}{(y_n)_{n\geq 1}}
\newcommand{\md}{\mathrm{d}}
\newcommand{\vac}{\varnothing}
\newcommand{\intent}[1]{[\![#1]\!]}
\title[Chaos for L\"uroth expansions]{Chaotic sets and Hausdorff dimension for L\"uroth expansions}
\author{Rafael Alcaraz Barrera}
\address{Instituto de F\'isica, Universidad Aut\'onoma de San Luis Potos\'i. Av. Manuel Nava 6, Zona Universitaria, C.P. 78290. San Luis Potos\'i, S.L.P. M\'exico}
\email{ralcaraz@ifisica.uaslp.mx}
\author{Gerardo Gonz\'alez Robert}
\email[Corresponding author]{gero@ciencias.unam.mx}

\date{\today}

\begin{document}

\maketitle
\begin{abstract}

\noindent We provide new similarities between regular continued fractions and L\"uroth series in terms of topological dynamics and Hausdorff dimension. In particular, we establish a complete analogue for the L\"uroth transformation of results by W. Liu, B. Li \cite{LiuLi2017} and W. Liu, S. Wang \cite{LiuWan2019} on the distal, asymptotic and Li-Yorke pairs for the Gauss map.

\end{abstract}

%\vspace{-2.3em}

\section{Introduction and statement of results}
\label{sec:intro}

\noindent Since the introduction of the term chaos by T. Li and J. Yorke in \cite{LiYor1975} (although without a formal definition), it has played a big role within the theory of topological dynamical systems. On the other hand, during the last 100 years or so, there has been significant interest in studying properties of different representations of real numbers under several mathematical perspectives; most notably: ergodic theory, fractal geometry, number theory and dynamical systems. 

\vspace{0.5em}Two particularly well-studied representations of numbers in the unit interval are \emph{regular continued fractions} (see \cite{DajKra2002,Khi1964,WanWu2008} and references therein) and \emph{L\"uroth series} (see \cite{DajKra1996,DajKra2002,Gal1976} and references therein). Both representations are obtained dynamically as the itineraries of a number $x$ under the \emph{Gauss map} and the \emph{L\"uroth map}, respectively. It is natural to ask if they share some dynamical or number theoretic properties.

\vspace{0.5em}The \emph{Gauss map} $\Ga:[0,1) \to [0,1)$ is given by:
\begin{equation}\label{eq:gaussmap}
\Ga(x) = \left\{
\begin{array}{clrr}      
\frac{1}{x} - \left[\frac{1}{x}\right] & \text{if} \quad  x \in (0,1);\\
0 & \text{if} \quad x = 0,&\\
\end{array}
\right.
\end{equation}
where $\left[ y \right]$ denotes the integer part of $y \in \R$. 

\vspace{0.5em}J. L\"uroth introduced  in \cite{Lur1883} a representation of real numbers now called \emph{L\"uroth series}. For each $x \in (0,1]$, let $a_1(x) \in \N$ be such that
\[
\frac{1}{a_1(x)}< x \leq \frac{1}{a_1(x)-1}
\]
(see \eqref{eq:luroth1}). Notice that $a_1(x) \geq 2$. The \emph{L\"uroth map} $\Lu:[0,1] \to [0,1]$ is given by:
\begin{equation}\label{eq:lurothmap1}
\Lu (x)=  
\left\{
\begin{array}{clrr}      
a_1(x)(a_1(x)-1)x - (a_1(x)-1)& \text{if} \quad  x \in (0,1],\\
0 & \text{if} \quad x = 0,&\\
\end{array}
\right.
\end{equation}  
We will delve deeper into the study of \eqref{eq:lurothmap1} in \cref{sec:prelim}.

\begin{figure}[ht!]
\centering
\begin{tikzpicture}[scale=5]
 \draw[dotted] (.5,0)--(.5,1);
  \draw[dotted] (.33,0)--(.33,1);
  \draw[dotted] (.25,0)--(.25,1);
 \draw[dotted] (0,0)--(1,1);
\draw(0,0)node[below]{\small $0$}--(.19,0)node[below]{\small $\frac15$}--(.255,0)node[below]{\small $\frac14$}--(.33,0)node[below]{\small $\frac13$}--(.5,0)node[below]{\small $\frac12$}--(1,0)node[below]{\small $1$}--(1,1)--(0,1)node[left]{\small $1$}--(0,0);
\draw[thick, black!50!black, smooth, samples =20, domain=.5:1] plot(\x,{1 / \x -1});
\draw[thick, black!50!black, smooth, samples =20, domain=.334:.5] plot(\x,{1 /\x -2});
\draw[thick, black!50!black, smooth, samples =20, domain=.25:.332] plot(\x,{1 /\x -3});
\draw[thick, black!50!black, smooth, samples =20, domain=.2:.25] plot(\x,{1 /\x -4});
\draw[thick, black!50!black, smooth, samples =20, domain=.167:.2] plot(\x,{1 /\x -5});
\draw[thick, black!50!black, smooth, samples =20, domain=.143:.167] plot(\x,{1 /\x -6});
\draw[thick, black!50!black, smooth, samples =20, domain=.125:.143] plot(\x,{1 /\x -7});
\draw[thick, black!50!black, smooth, samples =20, domain=.111:.125] plot(\x,{1 /\x -8});
\draw[thick, black!50!black, smooth, samples =20, domain=.101:.1112] plot(\x,{1 /\x -9});
\draw[thick, black!50!black, smooth, samples =20, domain=.091:.1] plot(\x,{1 /\x -10});
\filldraw[black!50!black] (0,0) rectangle (.09,1);
\end{tikzpicture}
\begin{tikzpicture}[scale=5]
\draw  (0,0) -- (.077,0) -- (.077,1) -- (0,1) -- cycle;
\draw[line width=0.3mm] (.077,0)--(.083,1)(.083,0)--(.09,1)(.09,0)--(.1,1)(.1,0)--(.11,1)(.11,0)--(.125,1)(.125,0)--(.143,1)(.143,0)--(.167,1)(.167,0)--(.2,1)(.2,0)--(.25,1)(.25,0)--(.33,1)(.33,0)--(.5,1)(.5,0)--(1,1);
\draw(0,0)node[below]{\small $0$}--(.19,0)node[below]{\small $\frac15$}--(.255,0)node[below]{\small $\frac14$}--(.33,0)node[below]{\small $\frac13$}--(.5,0)node[below]{\small $\frac12$}--(1,0)node[below]{\small $1$}--(1,1)--(0,1)node[left]{\small $1$}--(0,0);
\draw[dotted](.2,0)--(.2,1)(.25,0)--(.25,1)(.33,0)--(.33,1)(.5,0)--(.5,1);
\draw[dotted] (0,0)--(1,1);
\filldraw[black!50!black] (0,0) rectangle (.09,1);
\end{tikzpicture}
\caption{The Gauss and L\"uroth maps.}
\end{figure}
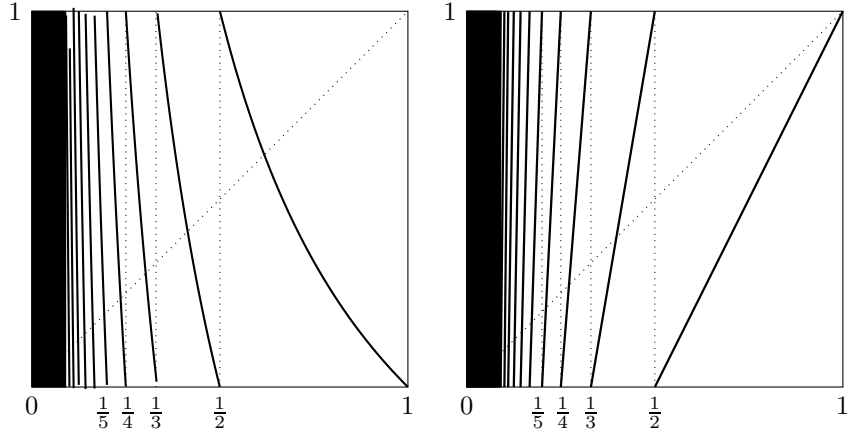

Both maps and their generated expansions have been studied under a dynamical systems' perspective. In particular, it is a well-known fact that $([0,1),\Ga)$ is ergodic with respect to the so-called \emph{Gauss measure} $\mu$ defined as follows: if $\leb$ denotes the Lebesgue measure on $[0,1]$, then
\begin{equation}
\mu(B) = \frac{1}{\log 2}\mathop{\int}\limits_{B} \frac{1}{1+x} \md\leb(x) \text{ for all Borel set } B\subseteq [0,1]
\end{equation} 
\cite[Section 4.2.4.]{ViaOli2016}.  On the other hand, H. Jager and C. de Vroedt in \cite{JagDev1969} showed that $([0,1], \Lu)$ is ergodic with respect to $\leb$. 

\vspace{0.5em}Let us illustrate some dynamical similarities between regular continued fraction expansions and L\"uroth series. Denote the regular continued fraction expansion of $x \in [0,1) \setminus \Q$ by 
\[
x = [0; c_1(x), c_2(x), \ldots];
\]
and the L\"uroth expansion of $x \in (0,1]$ by
\[
x = \left\langle a_1(x), a_2(x), \ldots \right\rangle.
\]
Using Birkhoff's Ergodic Theorem, it is straightforward to check that
\[
\mu\left(\set{x \in [0,1) \setminus \Q:  c_n(x) \leq M \, \text{ for some } M \in \N \text{ and for all } n \in \N}\right) = 0;
\]
and, if $\N_{\geq 2} := \set{n \in \N : n \geq 2}$,  
\begin{equation}\label{Eq:IntroLuBnd}
\leb\left(\set{x \in (0,1]: a_n(x) \leq M \, \text{ for some } M \in \N_{\geq 2} \text{ and for all } n \in \N}\right) = 0.
\end{equation}
Also,
\[
\mu\left(\set{x \in [0,1) \setminus \Q: \mathop{\lim}\limits_{n \to \f} c_n(x) = \infty}\right) = 0\]
and
\[\leb\left(\set{x \in (0,1]: \mathop{\lim}\limits_{n \to \f} a_n(x) = \infty}\right) = 0.
\]
Moreover, since $\mu$ is equivalent with $\leb$, we have 
\[
\leb\left(\set{x \in [0,1) \setminus \Q: c_n(x) \leq M \, \text{ for some } M \in \N \text{ and for all } n \in \N}\right) = 0
\]
and
\[
\leb\left(\set{x \in [0,1) \setminus \Q: \mathop{\lim}\limits_{n \to \f} c_n(x) = \infty}\right) = 0.
\]

\vspace{0.5em}V. Jarn\'ik showed in \cite{Jar1928} that
\begin{equation}\label{Eq:IntroJarnik}
\dimh\left(\set{x \in [0,1) \setminus \Q: c_n(x) \leq M \, \text{ for some } M \in \N \text{ and for all } n \in \N}\right) = 1.
\end{equation}
Here and thereafter, $\dimh$ stands for Hausdorff dimension. Later on, I.J. Good proved in \cite[Theorem 1]{Goo1941} that 
\begin{equation}\label{Eq:IntroGood}
\dimh\left(\set{x \in [0,1)\setminus\Q: \mathop{\lim}\limits_{n \to \f} c_n(x) = \infty}\right) =  \frac{1}{2}.
\end{equation}
Finally, it can be shown using elementary facts of iterated function systems (see \cite[Theorem 2.2.2]{BisPer2017}) that
\begin{equation}\label{Eq:IntroJarnikLuroth}
\dimh\left(\set{x \in (0,1]: a_n(x) \leq M \, \text{ for some } M \in \N_{\geq 2} \text{ and for all } n \in \N}\right) = 1
\end{equation}
and
\begin{equation}\label{Eq:IntroGoodLuroth}
\dimh\left(\set{x \in (0,1]: \mathop{\lim}\limits_{n \to \f} a_n(x) = \infty}\right) =  \frac{1}{2}. 
\end{equation}
Notice that \eqref{Eq:IntroJarnikLuroth} is a full analogue of \eqref{Eq:IntroJarnik} for the L\"uroth transformation and, similarly, \eqref{Eq:IntroGoodLuroth} of \eqref{Eq:IntroGood}. 
  
\subsection{Dynamical sets and chaotic properties}

Within the theory of topological and discrete dynamical systems, it is customary to consider a \emph{topological dynamical system} as a pair $(X,f)$ where $X$ is a compact (or pre-compact) metric space and $f:X \to X$ is a continuous transformation. In fact, the properties given in Definitions \ref{def:chaoticsets} and \ref{def:chaos1} below have been studied extensively in this setting \cite{BarVal2013,LiYe2016,Rue2017}. However, the results in the aforementioned references cannot be applied neither to the Gauss map nor the L\"uroth map, since both transformations are discontinuous. Nonetheless, Definitions \ref{def:chaoticsets} and \ref{def:chaos1} can be considered when $f$ is a piecewise continuous map of the unit interval with a countable number of discontinuities and $X$ is pre-compact.

\vspace{0.5em}As usual, given a transformation $f:X \to X$, $f^n(x)$ is the \emph{$n$-th iterate of $f$ at $x$}, i.e. $f \circ \ldots \circ f$, $n$ times. In particular, $f^1 = f$ and $f^0$ is the identity map on $X$. Also, the \emph{orbit of $x$ under $f$} is the set $\set{f^n(x)}_{n \geq 0}$.   

\begin{definition}\label{def:chaoticsets}
Given a dynamical system $(X,f)$ and $x, y \in X$ with $x \neq y$, we say that $(x,y)$ is:   
\begin{enumerate}[($i$)]
\item a \emph{proximal pair} if 
\[
\mathop{\liminf}\limits_{n \to \f} d(f^n(x), f^n(y)) = 0,
\]
\item an \emph{asymptotic pair} if 
\[
\mathop{\lim}\limits_{n \to \f} d(f^n(x), f^n(y)) = 0,
\]
\item a \emph{distal pair} if 
\[
\mathop{\liminf}\limits_{n \to \f} d(f^n(x), f^n(y)) > 0,
\]
\item a \emph{Li-Yorke pair} if $(x,y)$ is proximal and 
\[
\mathop{\limsup}\limits_{n \to \f} d(f^n(x), f^n(y)) > 0.
\]
\end{enumerate}

Also, given $x \in X$, we define the following \emph{dynamical sets}:

\begin{align*}
&\PR_f(x) := \set{y \in X: (x,y) \quad \text{is a proximal pair}}; \tag{$i$} \\
&\AS_f(x) := \set{y \in X: (x,y) \quad \text{is an asymptotic pair}}; \tag{$ii$}\\
&\noD_f(y) := \set{y \in X: (x,y) \quad \text{is a distal pair}}; \tag{$iii$}\\
&\LY_f(x) := \set{y \in X: (x,y) \quad \text{is a Li-Yorke pair}} \tag{$iv$}.
\end{align*}
\end{definition}

Clearly, the dynamical sets given in Definition \ref{def:chaoticsets} are symmetric in the following sense: $y \in P(x)$ if and only if $x \in P(y)$, where $P(y)$ is one of the dynamical sets given in Definition \ref{def:chaoticsets}. It readily follows from Definition \ref{def:chaoticsets} that

\begin{enumerate}[($i$)]
\item \[\AS_f(x) \subset \PR_f(x) \quad \text{and} \quad \LY_f(x) \subset \PR_f(x);\]
\item \[\noD_f(x) = X \setminus \PR_f(x)\] and; 
\item \[X = \AS_f(x) \dot \cup \LY_f(x) \dot \cup \noD_f(x),\] where $\dot \cup$ stands for disjoint union.
\end{enumerate}
Although our main results do not deal directly with proximal pairs, we have included them in the discussion for completeness.

\vspace{0.5em}Consider a compact (or precompact) metric space $X$ with metric $d$, and a transformation $f: X \to X$. We say that the dynamical system $(X,f)$ is \emph{topologically transitive} if for each pair of open and non-empty subsets $U, V \subset X$ there exists $n \in \N$ such that $f^n[U] \cap V \neq \vac$. We call a point $x\in X$ \emph{periodic} with respect to $f$ if there exists some $n\in\Na$ such that $f^n(x)=x$. Finally, we say that $(X,f)$ is \emph{sensitive to initial conditions} if there exists $\varepsilon > 0$ such that, for all $x \in X$ and for all $\delta > 0$ there are $y \in X$ with $d(x,y) < \delta$ and $n \in \N$ verifying $d(f^n(x), f^n(y)) \geq \varepsilon$.

\vspace{0.5em}Let us recall now the following notions of \emph{chaos}. 

\begin{definition}\label{def:chaos1}
Let $X$ a pre-compact metric space and $f: X \to X$ a transformation. We say that: 
\begin{enumerate}[($i$)]
\item the dynamical system $(X,f)$ is \emph{Devaney chaotic} if
\begin{enumerate}
\item $(X,f)$ is topologically transitive;
\item $(X,f)$ has dense periodic points;
\item $(X,f)$ is sensitive to initial conditions.
\end{enumerate}
\item the dynamical system $(X,f)$ is \emph{Li-Yorke chaotic} if there exists an uncountable set $S \subset X$ such that for all $x, y \in S$ with $x \neq y$ the pair $(x,y)$ is a Li-Yorke pair. Such set $S$ is called a \emph{scrambled set under $f$}.
\end{enumerate}
\end{definition}

%\vspace{0.5em}It is straightforward to check that if $(X,f)$ is locally eventually onto, then $(X,f)$ is topologically transitive. Also, W. Huang and X. Ye in \cite[Theorem 4.1]{HuaYe2002} shown that Devaney chaos imply Li-Yorke chaos when $f:X \to X$ is continuous. In the same setting, F. Blanchard et al. in \cite{BlaGlaKolMaa2002} shown that if the topological entropy of $(X,f)$ is positive, then $(X,f)$ is Li-Yorke chaotic. The dimensional properties of the sets given in Definition \ref{def:chaoticsets} have received attention recently --- see \cite{Neu2010}. % not only in the continuous case --- see for example \cite{LiChe2011,LiuLi2017,LiuWan2019}. 
\vspace{0.5em} W. Huang and X. Ye showed in \cite[Theorem 4.1]{HuaYe2002} that Devaney chaos implies Li-Yorke chaos when $f:X \to X$ is continuous. In the same setting, F. Blanchard et al. showed in \cite[Corollary 2.4]{BlaGlaKolMaa2002} that if the topological entropy of $(X,f)$ is positive, then $(X,f)$ is Li-Yorke chaotic. The dimensional properties of the sets given in Definition \ref{def:chaoticsets} have received attention recently --- see \cite{Neu2010}.

\subsection{Main results and structure of the paper}

Lately, W. Liu, B. Li in \cite{LiuLi2017} and W. Liu, S. Wang in \cite{LiuWan2019} studied chaotic properties of $([0,1), \Ga)$. The following statement summarizes \cite[Theorem 1.3, Corollary 1.4]{LiuLi2017} and \cite[Theorem 1.1, 1.2]{LiuWan2019}:

\begin{theorem}\label{thm:sumarised}
\leavevmode
\begin{enumerate}[($i$)]
\item For every $x \in [0,1)$ we have:
\begin{enumerate}[$a)$]
\item If $\displaystyle\liminf_{n\to\infty} \Ga^n(x)>0$, then $\AS_{\Ga}(x)$ is a countable set;
\item If $\displaystyle\lim_{n\to\infty} \Ga^n(x)=0$, then $\dimh \AS_{\Ga}(x)=\frac{1}{2}$;
\item If $\displaystyle\liminf_{n\to\infty} \Ga^n(x)=0$ and $\displaystyle\limsup_{n\to\infty} \Ga^n(x)>0$, then $\dimh(\AS_{\Ga}(x)) \leq \frac{1}{2}$. 
\end{enumerate}
\item For every $x\in [0,1)$, the set $\noD_{\Ga}(x)$ is dense in $[0,1)$. Moreover, $\leb(\noD_{\Ga}(x)) = 0$ and $\dimh(\noD_{\Ga}(x)) = 1$.
\item For every $x \in [0,1)$, we have that $\leb(\operatorname{LY}_{\Ga}(x)) = 1$.
\item The dynamical system $([0,1], \Ga)$ is Li-Yorke chaotic. Moreover, there exists a scrambled set under $\Ga$ of full Hausdorff dimension. 
\end{enumerate}
\end{theorem}

\vspace{0.5em}Our main results establish new similarities between the Gauss map and the L\"uroth map from a dynamical perspective. In fact, we obtain a full analogue of Theorem \ref{thm:sumarised} for the L\"uroth transformation. Although some of our strategies do follow  \cite{LiuLi2017, LiuWan2019}, our work is not a straightforward adaptation of their arguments. For instance, the computations used in the proof about asymptotic pairs for the Gauss map do not carry over into the Lüroth series context.

\begin{main}\label{TEO-ASYM-LUROTH}
Let $x$ be any number in $[0,1]$.
\begin{enumerate}[($i$)]
\item \label{Teo:Asintoticos-i} If $\displaystyle\liminf_{n\to\infty} \Lu^n(x)>0$, then $\AS_{\Lu}(x)$ is a countable set.
\item \label{Teo:Asintoticos-ii} If $\displaystyle\lim_{n\to\infty} \Lu^n(x)=0$, then $\dimh (\AS_{\Lu}(x))=\frac{1}{2}$.
\item \label{Teo:Asintoticos-iii} If $\displaystyle\liminf_{n\to\infty} \Lu^n(x)=0$ and $\displaystyle\limsup_{n\to\infty} \Lu^n(x)>0$, then $\dimh (\AS_{\Lu}(x))\leq \frac{1}{2}$.
\end{enumerate}
\end{main}

\begin{main}\label{TEO-DISTAL-LUROTH}
For every $x\in [0,1]$, the set $\noD_{\Lu}(x)$ is dense in $[0,1]$. Moreover, 
\[
\leb(\noD_{\Lu}(x)) = 0 \quad \text{and} \quad \dimh(\noD_{\Lu}(x)) = 1.
\] 
\end{main}
 
As a straightforward consequence of Theorems \ref{TEO-ASYM-LUROTH} and \ref{TEO-DISTAL-LUROTH} we obtain the following:

\begin{corollary}
For every $x\in [0,1]$, we have $\leb(\LY_{\Lu}(x)) = 1$.
\end{corollary}

Finally, we study the chaotic properties of the L\"uroth map.

\begin{main}\label{TEOREMA-DEVCHAOS-LUROTH}
The L\"uroth map $\Lu$ is Devaney chaotic.
\end{main}

\begin{main}\label{TEOREMA-SCRAM-LUROTH-01}
The L\"uroth map $\Lu$ is Li-Yorke chaotic. In fact, there exists a scrambled set $S$ under $\Lu$ which is dense and has full Hausdorff dimension. 
\end{main}

The paper is arranged as follows: in \cref{sec:prelim}, we provide the necessary background on L\"uroth series and Hausdorff dimension to perform our research. In \cref{sec:symbolic}, we provide a symbolic construction relevant to our purposes. We will concentrate our efforts in proving Theorems \ref{TEO-ASYM-LUROTH}, \ref{TEO-DISTAL-LUROTH}, \ref{TEOREMA-DEVCHAOS-LUROTH} and \ref{TEOREMA-SCRAM-LUROTH-01} in \cref{asintotico}, \cref{distal}, \cref{devaney}, and \cref{liyorke} respectively. Finally, in \cref{problems}, we pose some questions for further research.   

\subsection*{List of Notation}

In order to ease the reading, we list now the notation that we will use throughout our paper.

\vspace{0.5em}
\begin{itemize}
\item[-] We denote by $\N$ to the set of natural numbers and $\N_0$ stands for the set of non-negative integers. Also, given $m \in \N$, $\N_{\geq m} = \set{n \in \N : n \geq m}$.
\item[-] We write $\scD =\N_{\geq 2}$. 
\item[-] We denote the set of sequences in $\scD$ by $\scD^{\N}$.
\item[-] We denote the Lebesgue measure on $[0,1]$ by $\leb$.
%\item[-] Let $\leb$ denote the \emph{Lebesgue measure} on $[0,1]$ and let $\scB$ be the \emph{Borel $\sigma$-algebra of $[0,1]$}.
\item[-] For any $m,n\in\Z$ with $m\leq n$, we write 
\[
\intent{m,n}\colon=\{j\in\Z: m\leq j\leq n\}.
\]
\item[-] If $\scA$ is a non-empty set and $\bfa=\sanu\in\scA^\N$, we write \[\pref(\bfa,n)=(a_1,a_2,\cdots,a_n)\] for any $n\in\N$.
%\item[-] If $\scA$ is a non-empty set and $\bfa=\sanu\in\scA^\N$, we write \[\bfa\intent{m,n}=(a_m,a_{m+1},\cdots,a_n)\] for any $m,n\in\N$ satisfying $m\leq n$.
\item[-] If $\saxu$, $\sayu$ are sequences of positive real numbers, we write $x_n\ll y_n$ if there is some constant $C>0$ such that $x_n\leq C y_n$ for every large $n$. If $C$ depends on some parameter $\alpha$, we write $x_n\ll_{\alpha} y$. When $x_n\ll y_n$ and $y_n\ll x_n$ are true, we write $x_n\asymp y_n$.
\end{itemize}

\section{Preliminaries}
\label{sec:prelim}

\noindent In this section, we discuss some basic results on L\"uroth series and Hausdorff dimension. We will nevertheless assume some familiarity with both topics. The classical texts \cite{DajKra2002,Gal1976} discuss several aspects of L\"uroth series and other number theoretic maps. We refer the reader to \cite{BisPer2017,Fal2014} for a detailed account on the theory of Hausdorff dimension. Some newer results on L\"uroth series and Hausdorff dimension can be found in \cite{ArrGero2021,TanZho2021,TanZha2020}.

\subsection{Elements of L\"uroth Series}
\label{subsec:Lurprelim}
As mentioned previously, for any $x\in (0,1]$, we call $a_1(x)$ the unique number in $\Na_{\geq 2}$ determined by
\begin{equation}\label{eq:luroth1}
\frac{1}{a_1(x)}< x \leq \frac{1}{a_1(x)-1}.
\end{equation}
The \emph{L\"uroth map} $\scL:[0,1] \to [0,1]$ is given by
\[
\Lu (x)= 
\begin{cases}
a_1(x)(a_1(x)-1)x - (a_1(x)-1), \text{ if } x\in (0,1],\\
0, \text{ if } x=0.\\
\end{cases}
\]
Note that $\scL(x)=0$ if and only if $x=0$. Thus, given any $x\in(0,1]$, we can define the sequence of \emph{L\"uroth digits of $x$}, denoted $(a_n(x))_{n\geq 1}$ or $(a_n)_{n\geq 1}$ if there is no risk of ambiguity, by
\begin{equation}\label{eq:lurothdigits}
a_n(x) \colon= a_1(\scL^{n-1}(x)) \quad \text{for all} \quad n \in \N.
\end{equation}

\vspace{0.5em} In this case, the definition of $\Lu$ yields
\begin{equation}\label{Eq:Lu-01}
x= \frac{1}{a_1} + \frac{\scL(x)}{a_1( a_1-1)}.
\end{equation}
A repeated use of \eqref{Eq:Lu-01} replacing $x$ by the iterates of $\scL$ gives
%\begin{equation}\label{Eq:Lu-02}
%\begin{split}
%x  &= \frac{1}{a_1} + \frac{1}{a_1(a_1-1)a_2} + \frac{1}{a_1(a_1-1)a_2(a_2-1)a_3} + \ldots \\
%   &= \mathop{\sum}\limits_{n=1}^\f(a_n -1) \mathop{\prod}\limits_{j=1}^n \frac{1}{a_j(a_j-1)}.
%\end{split}
%\end{equation}
\begin{equation}\label{Eq:Lu-02}
x= \frac{1}{a_1} + \frac{1}{a_1(a_1-1)a_2} + \frac{1}{a_1(a_1-1)a_2(a_2-1)a_3} + \ldots.
\end{equation}
The series converges because the $n$-th term is positive and is bounded above by $2^{-n}$. We will write $x=\left\langle a_1,a_2,a_3,\ldots\right\rangle$ to mean \eqref{Eq:Lu-02}. %Moreover, using the same observation we obtain that the L\"uroth expansion of each $x \in (0,1]$ is well defined.

\vspace{0.5em}L\"uroth series allow us to identify the interval $(0,1]$ with a symbolic space. In what follows, we write 
\[
\scD^{\N} = \set{\saxu : \, x_n \in \scD \quad \text{for all } n \in \N},
\] 
i.e. the space of sequences in $\scD$. We endow $\scD$ with the discrete topology and $\scD^{\N}$ with the product topology. Additionally, we associate the dynamical system $((0,1], \Lu)$ with the symbolic dynamical system $(\scD^\N, \si)$ where $\si: \scD^\N \to \scD^\N$ is the so-called \emph{shift map} given by:
\[
\si(\sanu) = (a_{n+1})_{n \geq 1} \text{ for all } \sanu\in\scD^{\N}.
\]

\begin{proposition}\label{Teo:LuLambda}
The function $\Lambda:\scD^{\N}\to (0,1]$ given by 
\begin{equation}\label{eq:lambdamap}
\Lambda(\sanu) = \left\langle a_1,a_2,a_3,\ldots\right\rangle \; \text{ for all } \sanu\in\scD^{\N}
\end{equation}
is a semi-conjugacy between $(\scD^\N, \si)$ and $\left((0,1] , \Lu\right)$. That is, $\Lambda$ is a continuous and surjective function satisfying 
\begin{equation}\label{eq:semiconjlur}
\Lambda \circ\si = \Lu \circ \Lambda.
\end{equation}
\end{proposition}
\begin{proof}
It is shown in \cite[Section 4.3]{Gal1976} that $\Lambda$ is bijective. The continuity follows from the definition of the product topology. We notice that $\Lambda$ is not a homeomorphism, since $(0,1]$ is connected and $\scD^\N$ is totally disconnected. In order to check \eqref{eq:semiconjlur}, take any $\sanu$ in $\scD^\N$ and note that, by \eqref{Eq:Lu-02}, 
\[
\begin{split}
\Lu(\Lambda(\sanu)) &= \Lu(\left\langle a_1,a_2,a_3,\ldots\right\rangle) \\
&= a_1(a_1-1)(\left\langle a_1,a_2,a_2,\ldots\right\rangle) - (a_1-1)\\
&= (a_1 - 1) + \left\langle a_2,a_3,a_4,\ldots\right\rangle - (a_1-1)\\ 
&= \left\langle a_2,a_3,a_4,\ldots\right\rangle = \Lambda\circ \sigma(\sanu). 
\end{split}
\]
and the proposition follows.
\end{proof}

We should give a word of warning. We have followed J. Galambos \cite{Gal1976} rather than K. Dajani and C. Kraaikamp \cite{DajKra2002} when defining $\scL$. While our building blocks are intervals of the form $(\alpha,\beta]$, the other approach uses intervals of the form $[\alpha,\beta)$. We think that our choice is more natural, because the itinerary of any $x\in(0,1]$ with respect to $\scL$ provides the inverse function of $\Lambda$ and we do not have to mind the points whose orbit is eventually $0$. Other recent developments in Lüroth series, such as \cite{ArrGero2021,TanZho2021,TanZha2020}, also use Galambos' approach. However, the definition in \cite{DajKra2002} is precisely the one given by J. Lüroth in \cite{Lur1883}. In any case, this difference affects only countably many points and it is therefore irrelevant to our discussion. 

\vspace{0.5em} In view of Proposition \ref{Teo:LuLambda}, given any $n\in \N$ and $\bfc=(c_1,c_2,\ldots,c_n)\in\scD^n$, the set
\[
\scI_n(\bfc)\colon= \{ x=(0,1]: a_1(x)=c_1, \ldots, a_n(x)=c_n\}
\]
is non-empty. We refer to the sets $\scI_n(\bfc)$ as \emph{fundamental intervals of level $n$}, or just as \emph{fundamental intervals}. In particular, for any $c\in\scD$, the definition of $a_1$ translates into 
\[
\scI_1(c)=\left( \frac{1}{c}, \frac{1}{c-1}\right].
\]
For each $n\in\Na$, the fundamental intervals of order $n$ form a partition of $(0,1]$. We say that two fundamental intervals $I$ and $J$ of the same order are \textit{adjacent} if $I\neq J$ and $\inf\{|x-y|:x\in I,y\in J\}=0$. Hence, whenever $I$ and $J$ are not adjacent, there is at least one fundamental interval of the same order between them.

We denote by $\overline{\scI}_n(\bfc)$ the topological closure $\scI_n(\bfc)$. Given $n\in\N$ and $(c_1,\ldots, c_n)\in\scD^\N$, we write
\begin{equation}\label{eq:lurothpartialsum}
\left\langle c_1,c_2, \ldots, c_n\right\rangle \colon= 
\frac{1}{c_1} + \frac{1}{c_1(c_1-1)c_2} + \ldots  + \frac{1}{c_1(c_1-1)c_2 (c_2-1)\cdots c_{n-1}(c_{n-1}-1)c_n}.
\end{equation}

\vspace{0.5em}The following lemma is an exercise on mathematical induction.

\begin{lemma}\label{Le:Luroth-01}
\leavevmode
\begin{enumerate}[($i$)]
\item For any $n\in\N$ and any $\bfc=(c_1,\ldots,c_n)\in\scD^n$ we have
\[
\scI_n(\bfc) = \left( \left\langle c_1,c_2,\ldots, c_n\right\rangle, \left\langle c_1,c_2,\ldots, c_n-1\right\rangle\right].
\]
\item Every $x=\left\langle c_1,c_2,c_3,\ldots\right\rangle \in (0,1]$ and every $n\in\N$ satisfy
\[
\frac{1}{c_{n+1}}< \scL^n(x) \leq \frac{1}{c_{n+1}-1}.
\]
\end{enumerate}
\end{lemma}
The next corollary follows from the first part of Lemma \ref{Le:Luroth-01} and an inductive argument:
\begin{corollary}\label{Co:Luroth-02}
For any $n\in\N$ and any $\bfc=(c_1,\ldots, c_n)\in\scD^n$ we have
\[
\left| \scI_n(\bfc)\right| = \prod_{j=1}^n \frac{1}{c_j(c_j-1)};
\]
in particular, 
\[
|\scI_n(\bfc)| \leq \frac{1}{2^n}.
\]
\end{corollary}

%The following statement gives us a useful characterisation of the L\"uroth series of $x \in \Q \cap (0,1]$.

%\begin{proposition}\label{pr:lurothrationals}
%Let $x \in (0,1]$. Then $x \in \Q \cap (0,1]$ if and only if 
%\[
%\Lambda^{-1}(x) = \Lambda^{-1}(\left\langle a_1,a_2,a_3,\ldots\right\rangle) = \sanu
%\] 
%is a periodic or an eventually periodic sequence in $\scD^\N$.
%\end{proposition}

Let us name now the inverses of the different branches of each $\scL^n$. Given $n\in\N$ and $\bfc=(c_1,\ldots, c_n)\in\scD^n$, we define $\scT^n_{\bfc}:(0,1]\to\scI_n(\bfc)$ to be inverse of $\scL^n$ restricted to $\scI_n(\bfc)$. Hence, $\scT^n_{\bfc}$ is a linear function with slope
\[
\prod_{j=1}^n \frac{1}{c_j(c_j-1)}
\]
and its action on the Lüroth series of any $x=\langle a_1,a_2,a_3,\ldots\rangle\in(0,1]$ is as follows: 
\[
\scT_{\bfc}^n(\left\langle a_1,a_2,\ldots\right\rangle) 
= \left\langle c_1,c_2,\ldots, c_n,a_1,a_2,\ldots\right\rangle.
\]
We may extend the definition of $\scT_{\bfc}^{n}$ to an infinite sequence $\bfc$: if $\pref(\bfc,n)$ is the prefix of $\bfc$ of length $n$, then $\scT_{\bfc}^{n}\colon= \scT_{\pref(\bfc,n)}^{n}$.

\subsection{Hausdorff dimension} 

\noindent We suppose that the reader is acquainted with the notion of Hausdorff dimension. The next lemma, which we will use without any reference, follows from the definition of Hausdorff dimension (see Definition 1.2.1 and Proposition 1.2.6 in \cite{BisPer2017}).

\begin{lemma}
Let $A,B$ be subsets of $\R$. If $A\subseteq B$, then $\dimh A\leq \dimh B$. If $(A_n)_{n\geq 1}$ is a sequence of subsets of $\R$, then
\[
\dimh\left( \bigcup_{n\in \N} A_n\right) = \sup_{n\in\N} \dimh (A_n).
\]
\end{lemma}

For any non-empty sets $A,B \subseteq \R$, the \emph{diameter} of $A$ is
\[
|A|\colon=\sup\{| a - a'|:a,a'\in A\}
\]
and the \emph{distance} between $A$ and $B$ is
\[
d(A,B)\colon= \inf\left\{ |a-b|\colon a\in A, b\in B\right\}.
\]

\vspace{0.5em}Cantor sets are at the core of our proofs, so we need some general estimates for their Hausdorff dimension. Suppose $\clA=\{\clA_n:n\in\Na_0\}$ is a family of compact subsets of $\R$ with the following properties:

\begin{enumerate}[$(i)$]
\item The family $\clA_0$ contains exactly one element and each $\clA_n$ is finite,
\item Every $A\in\clA_n$ has positive diameter for all $n\in\Na$,
\item For every $n\in\N$ and every $A,B\in \clA_n$ either $A=B$ or $A\cap B=\vac$,
\item For every $n\in\N$ and every $B\in \clA_n$ there is some $A\in \clA_{n-1}$ such that $B\subseteq A$,
\item For every $n\in\N$ and every $A\in \clA_{n-1}$ there is some $B\in \clA_{n}$ such that $B\subseteq A$,
\item The quantity $d_n(\clA)\colon=\max\{|A|:A\in\clA_n\}$ tends to $0$ as $n$ tends to $+\infty$.
\end{enumerate}
The \emph{limit set} $\mathbf{A}_{\infty}$ of $\clA$ is 
\[
\mathbf{A}_{\infty}:= \bigcap_{n=0}^{\infty} \bigcup_{A\in\clA_n} A.
\]
For $n\in\N_0$ and any $A\in\clA_n$, the \emph{set of descendants of} $A$ is
\[
D(A)=\{B\in\clA_{n+1}: B\subseteq A\}.
\]
This construction is a simpler version of what D. Kleinbock and B. Weiss called \emph{strongly tree-like structure} in \cite{KleWei2010}. The next lemmas are proven in \cite[Section 3]{GeroGood}.

\begin{lemma}\label{Le:GJL01}
Let $\clA$ be as above. Assume that
\[
\liminf_{n\to\infty} \frac{\log(d_n(\clA)^{-1})}{n}>0,
\]
and that there is a sequence $(B_n)_{n\geq 1}$ in $\R_{>0}$ with the following properties:
\begin{enumerate}[$(i)$]
\item For any $n\in\Na_0$, any $A\in\scA_n$ and any $Y,Z\in D(A)$, the condition $Y\neq Z$ implies
\[
d(Y,Z)\geq B_n |A|,
\]
\item The next inequality holds:
\[
\limsup_{n\to\infty} \frac{\log\log(B_n^{-1})}{n} <1.
\]
\end{enumerate}
If $s>0$ is such that for every sufficiently large $n\in\N$ and every $A\in \clA_n$ we have
\begin{equation}\label{EcW1GJL}
\sum_{B\in D(A)} |B|^s\geq |A|^s,
\end{equation}
then $\dimh \bfA_{\infty}\geq s$.
\end{lemma}
\begin{lemma}\label{Le:GJL02}
Let $\clA$ be a family of compact sets as above except that we allow each $\clA_n$ to be at most countable. If $s>0$ is such that for every sufficiently large $n\in\N$ and any $A\in\scA_n$ we have
\[
\sum_{B\in D(A)} |B|^s\leq |A|^s,
\]
then $\dimh \bfA_{\infty}\leq s$. 
\end{lemma}

Let us consider an example. Take $N\in\Na_{\geq 3}$, define $\clA_0\colon=\{[0,1]\}$ and, for any $n\in\Na$, put
\[
\scA_n\colon=\left\{ \overline{\scI}_n(\bfa)\colon \bfa\in\intent{2,N}^n\right\}.
\]
The limit set of the family $\scA=\{\scA_n:n\in\Na_0\}$ is the Cantor set
\begin{equation}\label{Eq:DefFN}
F_N\colon=\left\{ x=\langle a_1,a_2,\ldots\rangle \in (0,1]\colon a_n\in\intent{2,N}^n \text{ for all } n\in\Na \right\}.
\end{equation}
Using Lemma \ref{Le:LuDisD} and that every $s\in(0,1)$ satisfies
\[
1=\sum_{b\geq 2} \frac{1}{b(b-1)} < \sum_{b\geq 2} \frac{1}{b^s(b-1)^s},
\]
we may obtain \eqref{Eq:IntroJarnikLuroth}. We omit the proof of this assertion, for it resembles that of Theorem \ref{TEO-DISTAL-LUROTH} (see \cref{distal}).

\section{A symbolic definition}
\label{sec:symbolic}
\noindent Proposition \ref{Teo:LuLambda} invites us to adopt a symbolic perspective. We now introduce the $R$-shuffle of two sequences. Briefly, given two sequences $\saxu$, $\sayu$ and an infinite set $R\subseteq \Na$ with infinite complement, the $R$-shuffle of $\saxu$ and $\sayu$ is the result of placing $\sayu$ on $R$ and $\saxu$ on $\Na\setminus R$.

\begin{definition}
Let $R\subseteq \N$ be an infinite set such that $Q=\N\setminus R$ is infinite. Write $R=\{r_n:n\in\N\}$ and $Q=\{q_n:n\in\N\}$ with $r_n<r_{n+1}$, $q_n<q_{n+1}$ for all $n\in\N$. Let $\scA\subseteq \scD$ be a non-empty set. The $R$\emph{-shuffle} of two sequences $\bfx=\saxu,\bfy=\sayu\in \scA^{\N}$ is the sequence $\scB(\bfx,\bfy;R)=(c_j)_{j\geq 1}\in\scA^{\N}$ given by
\[
c_{q_k}=x_k, \; c_{r_k}=y_k \text{ for all } k\in\N.
\]
If there is some $m\in\N_{\geq 2}$ such that $R=\{km:k\in\N\}$, we write $\scB_m(\bfa,\bfb)=\scB(\bfa,\bfb;R)$. We extend this notation for certain finite sequences: for $m,n\in\N$ with $m\geq 2$, $\bfa=(a_j)_{j\geq 1}\in \scA^{(m-1)n}$, and $\bfb=(b_j)_{j\geq 1}\in \scA^n$, the sequence $\scB_m(\bfa,\bfb)=(c_j)_{j\geq 1} \in \scA^{mn}$ is given by 
\[
\scB_m(\bfa,\bfb) = (a_1,\ldots, a_{m-1}, b_1,a_{m}, \ldots, b_{n-1},a_{(n-1)(m-1)+1}, \ldots,  a_{n(m-1)}, b_n).
\]
In other words, we truncate the corresponding definition for infinite sequences.
\end{definition}
\section{Asymptotic pairs: proof of Theorem \ref{TEO-ASYM-LUROTH}}
\label{asintotico}

\noindent Take an arbitrary number $x\in[0,1]$. We prove the three parts of Theorem \ref{TEO-ASYM-LUROTH} separately. Assume first that $\mathop{\liminf}\limits_{n \to \f} \scL^n(x)>0$. Then, $x\in (0,1]$ and, writing $x=\left\langle a_1,a_2,a_3,\ldots\right\rangle$, we have 
\[
\mathop{\limsup}\limits_{n \to \f} a_n<+\infty
\]
(see Lemma \ref{Le:Luroth-01}); that is, $\sanu$ is a bounded sequence. Consider $y\in\AS_{\scL}(x)$. Clearly, $y\neq 0$ and we can write $y=\left\langle b_1,b_2,b_3,\ldots\right\rangle$. If $\sabu$ were unbounded, there would be a strictly increasing sequence of natural numbers $(n_j)_{j\geq 1}$ such that $\scL^{n_j}(y)\to 0$ when $j\to \infty$. We would then obtain a contradiction:
\begin{align*}
0 &= \lim_{n\to\infty} \left| \scL^n(x) - \scL^n(y)\right| \nonumber\\
& = \limsup_{j\to\infty} \left| \scL^{n_j}(x)-\scL^{n_j}(y)\right| = \limsup_{j\to\infty} \scL^{n_j}(x) \geq \liminf_{n\to\infty} \scL^n(x)>0. \nonumber
\end{align*}
Therefore, $\sabu$ is a bounded sequence. Put 
\[
M\colon=\max\left\{\max_{n\in\Na} a_n, \max_{n\in\Na} b_n\right\}.
\]
Then, the orbits under $\scL$ of $x$ and $y$ are thus contained in the Cantor set
\[
F_M=
\{\left\langle c_1,c_2,c_3,\ldots\right\rangle: c_n\in\intent{2,M} \text{ for all } n\in\Na \}.
\]
For each $r\in\intent{2,M}$ define 
\[
\mathscr{K}(r) = \bigcup_{j=2}^M \overline{\scI}_2(r,j).
\]
Observe that the collection of mutually disjoint compact sets $\{\mathscr{K}(r):r\in\intent{2,M}\}$ is finite, hence
\[
\veps \colon= \min\left\{d\left(\mathscr{K}(r), \mathscr{K}(s)\right)\colon s,r\in\intent{2,M}, s\neq r\right\}>0.
\]
Let $N=N(\veps)\in\N$ be such that $|\scL^n(x)-\scL^n(y)|<\veps$ for $n\geq N$. For such $n$, the points $\scL^n(x)$ and $\scL^n(y)$ belong to the same set $\mathscr{K}(r)$ and hence $a_{n+1}=b_{n+1}$. Moreover, it is obvious that every number $z=\left\langle c_1,c_2,c_3,\ldots\right\rangle\in(0,1]$ such that $c_n=a_n$ holds for every sufficiently large $n$ belongs to $\AS_{\scL}(x)$. This shows that
\[
\textstyle\AS_{\scL}(x) = \displaystyle\left\{ y=\left\langle b_1,b_2,b_3,\ldots\right\rangle \colon b_n=a_n \text{ for sufficiently large } n\right\},
\]
so $\AS_{\scL}(x)$ is countable.

Second, assume that $\displaystyle\lim_{n\to\infty} \scL^n(x) = 0$. If $y\in \AS_{\scL}(x)$, then $\scL^n(y)\to 0$ if $n\to \infty$, so either $y=0$ or $y=\left\langle b_1,b_2,b_3,\ldots\right\rangle$ satisfies $b_n\to\infty$ as $n\to \infty$ (see Lemma \ref{Le:Luroth-01}). Conversely, every $y$ with any of these two properties is an element of $\AS_{\scL}(x)$. Therefore,
\[
\textstyle\AS_{\scL}(x) = \left\{ y=\left\langle b_1,b_2,b_3,\ldots\right\rangle \colon \displaystyle\lim_{n\to\infty} b_n = \infty\right\}\cup\{0\}
\]
and, by \eqref{Eq:IntroGoodLuroth}, $\dimh \AS_{\scL}(x)=\frac{1}{2}$.

Third, let us assume that $\mathop{\liminf}\limits_{n \to \f}\scL^n(x)=0$ and that $\mathop{\limsup}\limits_{n \to \f} \scL^n(x)>0$. The second inequality yields $x\neq 0$, so we can write $x=\left\langle a_1,a_2,a_3,\ldots\right\rangle$. The argument in this case is much more involved. In a few words, we are going to build a superset of $\AS_{\scL}(x)$ whose Hausdorff dimension we will estimate with Lemma \ref{Le:GJL02}.

\begin{proposition}\label{Propo:Asymp01}
Given $y=\left\langle b_1,b_2,\ldots\right\rangle$, $M\in\N_{\geq 3}$, and $n\in\N$, let $P(x,y,M,n)$ be the statement
\[
P(x,y,M,n)\colon= (a_{n+1}\leq M\implies a_{n+1}=b_{n+1})\land(a_{n+1}> M\implies b_{n+1}>M).
\]
The following holds:
\[
\textstyle\AS_{\scL}(x)
\subseteq\displaystyle
\bigcap_{M\geq 3}\bigcup_{N\in\N}\bigcap_{n\geq N} \left\{y \in (0,1]\colon P(x,y,M,n) \right\}.
\]
\end{proposition}
Proposition \ref{Propo:Asymp01}, and in particular $P$, provides a precise statement of the following claim: if $y\in \AS_{\scL}(x)$, then the digits in the Lüroth expansions of $x$ and $y$ which are smaller than a given bound eventually coincide.
\begin{proof}
Let $M_1\colon (0,24^{-1})\to\N_{\geq 2}$ be the function that associates to each $\veps\in (0,24^{-1})$ the integer
\[
M_1(\veps) \colon= \max\left\{ m\in\N\colon m(m+1)\leq (4\veps)^{-1} \right\}.
\]
It is easy to check that $M_1$ is non-increasing, surjective, and that $M_1(\veps)\to\infty$ when $\veps\to 0$.

Take $y\in\AS_{\scL}(x)$. As before, $y\neq 0$ and we write $y=\left\langle b_1,b_2,\ldots\right\rangle$. Take an arbitrary $M\in\N_{\geq 3}$. We now obtain a natural number $N$ such that $P(x,y,M,n)$ holds for any $n\in\N_{\geq N}$. Let $\veps>0$ satisfy $M_1(\veps)=M$ and let $N\in\N$ be such that 
\[
\left| \scL^n(x) - \scL^n(y)\right| < \veps 
\quad
\text{ for all } n\in \Na_{\geq N}.
\]
Pick any $n\in\N_{\geq N}$. 

Assume first that $a_{n+1} \leq M$. We will conclude $a_{n+1}=b_{n+1}$ after showing that $a_{n+1}\neq b_{n+1}$ leads to the contradiction
\begin{equation}\label{Eq:DemAsymp01}
|\scL^{n}(x)-\scL^n(y)|\geq \veps.
\end{equation}
The argument amounts to check that $b_{n+1}$ cannot belong to any of the sets $\intent{2,a_{n+1}-2}$, $\{a_{n+1}-1\}$, $\intent{a_{n+1}+2,\infty}$, or $\{a_{n+1}+1\}$. It may happen that some of these sets are empty, but there is nothing to prove if this is true.

\paragraph{Case 1.1} If $b_{n+1}\leq a_{n+1} - 2$, then 
\begin{align*}
\left| \scL^n(y) - \scL^n(x)\right| &=  \scL^n(y) - \scL^n(x) \\
&> \frac{1}{b_{n+1}} - \frac{1}{a_{n+1}-1} \\ 
&\geq \frac{1}{a_{n+1}-2} - \frac{1}{a_{n+1}-1}= \frac{1}{(a_{n+1}-1)(a_{n+1}-2)}> \frac{1}{M(M+1)}\geq \veps.
\end{align*}
Therefore, $b_{n+1}\geq a_{n+1}-1$.

\paragraph{Case 1.2} Assume that $b_{n+1}= a_{n+1}-1$. We will proceed in three steps: firstly, we show $a_{n+2}=2$; afterwards, we show $b_{n+2}\geq 4$; finally, we obtain a contradiction.

\textit{First step.} If we had $a_{n+2}\geq 3$, then $\scL^{n+1}(x) \leq 2^{-1}$ would hold and
\[
\scL^n(x) 
= \frac{1}{a_{n+1}} + \frac{\scL^{n+1}(x)}{a_{n+1}(a_{n+1}+1)}
\leq \frac{1}{a_{n+1}} + \frac{1}{2a_{n+1}(a_{n+1}+1)}.
\]
The previous inequality and $\scL^n(y)> b_{n+1}^{-1}=(a_{n+1}-1)^{-1}$ yield
\begin{align*}
\left| \scL^n(y) - \scL^n(x)\right| &=  \scL^n(y) - \scL^n(x) \\
&> \frac{1}{a_{n+1}-1} - \left(\frac{1}{a_{n+1}}+ \frac{1}{2a_{n+1}(a_{n+1}-1)}\right) \\ 
&= \frac{1}{2a_{n+1}(a_{n+1}-1)}> \frac{1}{2M (M + 1)}\geq \veps.
\end{align*}
We conclude that $a_{n+2}=2$.

\textit{Second step.} If $b_{n+2}\in\{2,3\}$ were true, we would have $\scL^{n+1} (y)> 3^{-1}$ and
\[
\scL^n(y) 
= \frac{1}{a_{n+1}-1} + \frac{\scL^{n+1}(y)}{(a_{n+1}-1)(a_{n+1}-2)}
\geq \frac{1}{a_{n+1}-1} + \frac{1}{3(a_{n+1}-1)(a_{n+1}-2)}.
\]
The previous expression and $\scL^n(x)<(a_{n+1}-1)^{-1}$ give
\begin{align*}
\left| \scL^n(y) - \scL^n(x)\right| &=  \scL^n(y) - \scL^n(x) \\
&> \frac{1}{a_{n+1}-1} + \frac{1}{3(a_{n+1}-1)(a_{n+1}-2)} - \frac{1}{a_{n+1}-1} \\ 
&= \frac{1}{3(M + 1) M} \geq \veps.
\end{align*}
We conclude that $b_{n+2}\geq 4$.

\textit{Third Step.} By the first step, $\scL^{n+1}(x)>2^{-1}$ and, by the second step, $\scL^{n+1}(y)\leq 3^{-1}$; hence, 
\[
\scL^{n+1}(x) - \scL^{n+1}(y) > \frac{1}{2} - \frac{1}{3} = \frac{1}{6}>\veps.
\]
We conclude that $a_{n+1}\leq b_{n+1}$.

\paragraph{Case 1.3} Assume that $b_{n+1}\geq a_{n+1}+ 2$, then 
\[
\scL^n(y) \leq \frac{1}{b_{n+1}-1}\leq \frac{1}{a_{n+1}+1}.
\]
The previous inequality together $\scL^n(x)> a_{n+1}^{-1}$ give
\begin{align*}
\left| \scL^n(y) - \scL^n(x)\right| &=  \scL^n(x) - \scL^n(y) \\
&> \frac{1}{a_{n+1}} - \frac{1}{a_{n+1}+1} = \frac{1}{a_{n+1}(a_{n+1}+1)} \geq \frac{1}{(M + 1)M } > \veps.
\end{align*}
Hence, we have $b_{n+1} \in \{a_{n+1}, a_{n+1}+1\}$.

\paragraph{Case 1.4} Assume that $b_{n+1}= a_{n+1}+1$. The proof is quite similar to that of Case 1.2. 

\textit{First step.} If we had $a_{n+2}\in \{2,3\}$, then $\scL^{n+1}(x)>3^{-1}$ would hold and
\[
\scL^n(x) 
= \frac{1}{a_{n+1}} + \frac{\scL^{n+1}(x)}{a_{n+1}(a_{n+1}+1)}
> \frac{1}{a_{n+1}} + \frac{1}{3a_{n+1}(a_{n+1}+1)}.
\]
This inequality and $\scL^n(y)\leq (b_{n+1}-1)^{-1}=a_{n+1}^{-1}$ imply
\begin{align*}
\left| \scL^n(y) - \scL^n(x)\right| &=  \scL^n(x) - \scL^n(y) \\
&> \frac{1}{a_{n+1}} + \frac{1}{3a_{n+1}(a_{n+1}-1)} - \frac{1}{a_{n+1}}  \\ 
&= \frac{1}{3a_{n+1}(a_{n+1}-1)} > \frac{1}{3M (M + 1)}\geq \veps.
\end{align*}
Thus, we must have $a_{n+2}\geq 4$.

\textit{Second step.} If $b_{n+2}\geq 3$, then $\scL^{n+1}(y)\leq 2^{-1}$ and
\[
\scL^n(y) 
= \frac{1}{b_{n+1}} + \frac{\scL^{n+1}(y)}{b_{n+1}(b_{n+1}-1)}
\leq \frac{1}{a_{n+1}+1} + \frac{1}{2a_{n+1}(a_{n+1}+1)}.
\]
The last inequality and $\scL^n(x)> a_{n+1}^{-1}$ lead the contradiction
\begin{align*}
\left| \scL^n(y) - \scL^n(x)\right| &=  \scL^n(x) - \scL^n(y) \\
&> \frac{1}{a_{n+1}} - \left( \frac{1}{a_{n+1}+1} + \frac{1}{2a_{n+1}(a_{n+1}+1)}\right) \\ 
&=\frac{1}{2a_{n+1}(a_{n+1}+1)} \geq \frac{1}{2(M + 1)M} \geq \veps.
\end{align*}
Hence $b_{n+2}=2$.

\textit{Third Step.} The first two steps give $\scL^{n+1}(x)\leq 3^{-1}$ and $\scL^{n+1}(y)> 2^{-1}$, so
\[
\scL^{n+1}(y) - \scL^{n+1}(x) > \frac{1}{2} - \frac{1}{3} = \frac{1}{6}>\veps
\]
and we may say that $b_{n+1}+1\neq a_{n+1}$. Combining Cases 1.1, 1.2, 1.3, and 1.4, we conclude that $a_{n+1}\leq M$ implies  $a_{n+1}=b_{n+1}$.

Now assume that $a_{n+1}>M$. If we had $b_n\leq M$, the previous argument and the symmetry of the asymptotic relation guarantee $a_{n+1}=b_{n+1}\leq M$, a contradiction. Therefore, $b_{n+1}>M$. 

Since $n\in\Na_{\geq N}$ was arbitrary, the proof is done.
\end{proof}
For each $M\in\Na_{\geq 3}$ and $N\in\Na$, we define the functions $f_N^M, g_N^M:\Na\to\scD\cup\{\infty\}$ as follows: for any $n\in\Na$
\[
f_N^M(n) = 
\begin{cases}
a_n, \text{ if } 1\leq n\leq N-1,\\
a_n, \text{ if } N\leq n \text{ and } a_n\leq M,\\
M+1, \text{ if } N\leq n \text{ and } M<a_n,
\end{cases} \quad
g_N^M(n) = 
\begin{cases}
a_n, \text{ if } 1\leq n\leq N-1,\\
a_n, \text{ if } N\leq n \text{ and } a_n\leq M,\\
+\infty, \text{ if } N\leq n  \text{ and } M<a_n.
\end{cases}
\]
Also, we define the set
\[
G_N^M\colon= \left\{ y=\langle b_1,b_2,b_3,\ldots\rangle \in (0,1]: f_N^M(n)\leq b_n\leq g_N^M(n) \text{ for all } n\in\Na_{\geq N}\right\}.
\]
With this notation, Proposition \ref{Propo:Asymp01} can be summarized as follows:
\begin{equation}\label{Eq:Asymp01}
\textstyle\AS_{\scL}(x)
\displaystyle
\subseteq \bigcap_{M\geq 3}\bigcup_{N\in\N} G_N^M.
\end{equation}
\begin{proposition}\label{Propo:Asymp02}
For each $\veps>0$ there is some $M\in\N_{\geq 3}$ such that every $N\in\N$ verifies $\dimh (G_N^M) < \frac{1}{2}+\veps$. 
\end{proposition}
\begin{proof}
Take $M\in\N_{\geq 3}$. We first show that $(\dimh (G_N^M))_{N\geq 1}$ is a constant sequence. On the one hand, $(\dimh (G_N^M))_{N\geq 1}$ is non-decreasing because so is $(G_N^M)_{N\geq 1}$. On the other hand, let $N\in \N_{\geq 2}$ be arbitrary and note that
\begin{equation}\label{Eq:Asymp02}
G_N^M = \bigcup_{\bfc\in \clD^{N-1}} G_N^M \cap \scI_{N-1}(\bfc).
\end{equation}
For any $\bfc,\bfd\in\scD^{N-1}$, the function 
\[
\scR_{\bfc}^{\bfd}\colon G_N^M \cap \scI_{N-1}(\bfc)\to G_N^M \cap \scI_{N-1}(\bfd)
\]
given by
\[
\scR_{\bfc}^{\bfd}(y)=\scT_{\bfd}^{N-1}\circ \scL^{N-1}(y)
\quad
\text{ for all } y\in G_N^M \cap \scI_{N-1}(\bfc)
\]
is bi-Lipschitz as it is a composition of linear functions. From a symbolic perspective, $\scR_{\bfc}^{\bfd}$ replaces the prefix $\bfc=(c_1,\ldots,c_{N-1})$ with $\bfd=(d_1,\ldots,d_{N-1})$; in other words, for any 
\[
\langle c_1,\ldots,c_{N-1},b_N,b_{N+1},\ldots\rangle\in G_N^M\cap\scI_{N-1}(\bfc)
\]
we have
\[
\scR_{\bfc}^{\bfd}\left(\left\langle c_1,\ldots, c_{N-1}, b_N,b_{N+1}, b_{N+2}, \ldots\right\rangle\right)
=
\left\langle d_1,\ldots, d_{N-1}, b_N,b_{N+1}, b_{N+2}, \ldots\right\rangle.
\]
Therefore, $\dimh (G_N^M \cap \scI_{N-1}(\bfc))=\dimh (G_N^M \cap \scI_{N-1}(\bfd))$ (see \cite[Proposition 3.3]{Fal2014}) and, in view of \eqref{Eq:Asymp02},
\begin{equation}\label{Eq:Asymp03}
\dimh \left(G_N^M \right)= \dimh \left(G_N^M\cap \scI_{N-1}(a_1,\ldots, a_{N-1})\right).
\end{equation}
Then, using
\[
G_N^M\cap \scI_{N}(a_1,\ldots, a_{N}) =  G_{N+1}^M\cap \scI_N(a_1,\ldots, a_N),
\]
we have
\begin{align*}
\dimh \left( G_N^M\right)&= \dimh \left(G_{N}^M \cap \scI_{N-1}(a_1,\ldots, a_{N-1})\right)\\
&\geq \dimh \left(G_N^M \cap \scI_N(a_1,\ldots,a_{N-1}, a_N)\right)\\
&=\dimh \left(G_{N+1}^M \cap \scI_{N}(a_1, \ldots, a_{N-1},a_{N})\right)\\
&= \dimh \left(G_{N+1}^M \right).
\end{align*}
The argument also holds for $N=1$ if we replace $\scI_{N-1}(a_1,\ldots, a_{N-1})$ with $[0,1]$. Hence, $(\dimh (G_N^M))_{N\geq 1}$ is a constant sequence.

Take $\veps>0$ and put $s=\frac{1}{2}+\veps$. Let $M\in\N$ be such that
\begin{equation}\label{Eq:DemAsymp02}
M\geq 1 + \left( \frac{1}{2\veps} \right)^{\frac{1}{2\veps}}
\end{equation}
and take any $N\in\N$. We now get an upper estimate for $\dimh \left(G_N^M\cap \scI_{N-1}(a_1,\ldots, a_{N-1})\right)$ using Lemma \ref{Le:GJL02}. To each $n\in\N$ and any $\bfc=(c_1, \ldots, c_n)\in \scD^n$ such that $f(N-1+j)\leq c_j\leq g(N-1+j)$ for $j\in\intent{1,n}$, we associate the compact set
\[
A_n(\bfc)\colon= \overline{\scI}_{N+n-1}(a_1,\ldots, a_{N-1}, c_1, \ldots, c_n)
\]
and we call $\scA_n$ the collection of such sets $A_n(\bfc)$. Define $\scA_0=\{[0,1]\}$. The limit set $\bfA_{\infty}(M,N)$ of the family $\scA=\{\scA_n:n\in\Na_0\}$ is precisely $G_{N}^M\cap \scI_{N-1}(a_1,\ldots, a_{N-1})$. 

We claim that for any $n\in\N$ and any valid $\bfc=(c_1,\ldots, c_n)$ the following holds:
\begin{equation}\label{Eq:Asymp04}
\sum_{k=f(n+1)}^{g(n+1)} |A_{n+1}(\bfc k)|^s \leq |A_n(\bfc)|^s.
\end{equation}
Indeed, the inequality above is, by the definition of $A_n(\bfc)$ and Corollary \ref{Co:Luroth-02}, equivalent to
\begin{equation}\label{Eq:Asymp05}
\sum_{k=f(n+1)}^{g(n+1)} \frac{1}{k^s(k-1)^s} \leq 1.
\end{equation}
If $g(n+1)<+\infty$, then \eqref{Eq:Asymp05} reduces to $a_{n+1}^{-s}(a_{n+1}-1)^{-s}\leq 1$, which is true because $s>0$ and $a_{n+1}\geq 2$. Now assume that $g(n+1)=+\infty$. From
\begin{align*}
\sum_{k=f(n+1)}^{g(n+1)} \frac{1}{k^s(k-1)^s} &= \sum_{k=M+1}^{\infty} \frac{1}{k^s(k-1)^s} \\
&\leq \sum_{k=M+1}^{\infty} \frac{1}{(k-1)^{2s}} \\
&\leq \int_{M-1}^{\infty} x^{-2s}\md x = \frac{(M-1)^{-2\veps}}{2\veps}
\end{align*}
and \eqref{Eq:DemAsymp02} we obtain $\frac{(M-1)^{-2\veps}}{2\veps}\leq 1$. Then, \eqref{Eq:Asymp05} holds and, by Lemma \ref{Le:GJL02} and \eqref{Eq:Asymp03}, we conclude
\[
\dimh \left(G_N^M\right)= \dimh \left(\bfA_{\infty}(M,N)\right) \leq s = \frac{1}{2} + \veps.
\]
\end{proof}
We can now finish the proof of the third point in Theorem \ref{TEO-ASYM-LUROTH}. For any $\veps>0$ let $M=M(\veps)$ be as in Proposition \ref{Propo:Asymp02}. Then,
\[
\dimh \left(\bigcup_{N\in\N} G_N^M \right)\leq \frac{1}{2} + \veps;
\]
so, by \eqref{Eq:Asymp01}, $\dimh(\AS_{\scL}(x))\leq \frac{1}{2}+\veps$. Since $\veps>0$ was arbitrary, $\dimh (\AS_{\scL}(x))\leq \frac{1}{2}$ and the proof of Theorem \ref{TEO-ASYM-LUROTH} is complete.

\section{Distal pairs: proof of Theorem \ref{TEO-DISTAL-LUROTH}}
\label{distal}

\noindent Let $x\in [0,1]$ be arbitrary. We consider two different cases: $x=0$ and $x\in (0,1]$. 

The first case is already done. In fact, since $\scL(0)=0$, a real number $y$ belongs to $\noD_{\scL}(0)$ if and only if 
\[
\liminf_{n\to\infty} \scL^n(y)>0,
\]
or, equivalently, if its Lüroth digits are bounded. Hence, by \eqref{Eq:IntroLuBnd} and \eqref{Eq:IntroJarnikLuroth}, we conclude $\leb(\noD_{\scL}(0))=0$ and $\dimh \noD_{\scL}(0)=1$.

Now assume that $x\in(0,1]$ and write $x=\left\langle a_1,a_2,a_3,\ldots\right\rangle$. Both $x$ and $\sanu$ will remain fixed for the rest of the proof. We first show that $\leb(\noD_{\scL}(x))=0$ relying on a result due to W. Philipp. Afterwards, in order to show $\dimh \noD_{\scL}(x)=1$, we construct a subset of $\noD_{\scL}(x)$ with full Hausdorff dimension. 

\subsection{The distal pairs of $x$ are null}

W. Philipp proved in the 1967 article \cite{Phi1967} a quantitative version of the Borel-Canelli Lemma (Theorem \ref{Teo:PhilipBC} below). W. Philipp used it to obtain measure theoretic properties of numeration systems including regular continued fractions and expansions on integer base $b\geq 2$. In what follows, $\mathscr{O}$ refers to Landau's big-O notation. 
\begin{theorem}\cite[Theorem 3]{Phi1967}\label{Teo:PhilipBC}
Let $(X,\scB,\mu)$ be a probability space and let $(E_n)_{n\geq 1}$ be a sequence of measurable sets. For each $N\in\Na$ and each $t\in X$, define
\[
A(N,t)\colon =\#\left\{ n\in \intent{1,N}\colon t\in E_n\right\}
\]
and 
\[
\vphi(N)\colon = \sum_{n=1}^N \mu(E_n).
\]
Suppose that there is a summable sequence of non-negative real numbers $(C_n)_{n\geq 1}$ such that for any $m,n\in\N$ satisfying $m<n$ we have
\[
\mu(E_n\cap E_m)\leq \mu(E_n)\mu(E_m)+\mu(E_n)C_{n-m}.
\]
Then, for $\mu$-almost every $t\in X$ and every $\veps>0$ we have
\[
A(N,t)= \vphi(N) + \mathscr{O}\left(\sqrt{\vphi(N)}\left(\log\vphi(N)\right)^{\frac{3}{2}+\veps}\right).
\]
\end{theorem}
We follow rather closely W. Philipp's steps and apply Theorem \ref{Teo:PhilipBC} on Lüroth series to conclude the next result:
\begin{proposition}\label{Teo:PhilippLuroth}
Let $(I_n)_{n\geq 1}$ be a sequence of intervals in $[0,1]$. For each $N\in\N$ and each $t\in[0,1]$ define 
\[
A(N,t):=\#\{n\in\intent{1,N}:\scL^n(t)\in I_n\} \quad\text{and}\quad
\vphi(N):=\sum_{n=1}^N \leb(I_n).
\]
Then, almost every $t\in [0,1]$ verifies the following statement: for all $N \in \N$ and all $\veps>0$ we have
\[
A(N,t) = \vphi(N) + \mathscr{O}\left( \vphi(N)^{\frac{1}{2}}  (\log \vphi(N) )^{\frac{3}{2}+\veps} \right).
\]
\end{proposition}

%We must estimate the measure of the overlap of certain sets for the proof of Theorem \ref{Teo:PhilippLuroth}.
\begin{lemma}\label{Le:Dist-A}
Let $\scI_k$ be a fundamental interval of a given order $k\in\Na$. Then, for every interval $(\alpha,\beta]\subseteq (0,1]$ and every $n\in\Na_{\geq k}$, we have
\[
\leb\left( \scL^{-n}\left[(\alpha,\beta]\right]\cap \scI_k\right)=\leb\left( (\alpha,\beta] \right)\leb(\scI_k).
\]
\end{lemma}
\begin{proof}
See (3.4) in \cite{JagDev1969}.
\end{proof}
\begin{lemma}\label{Le:Dist-B}
Let $I,J\subseteq [0,1]$ be two intervals, $n\in\N$ and let $[\frac{n}{2}]$ be the integer part of $\frac{n}{2}$, then
\[
\leb\left( I\cap \Lu^{-n}[J]\right) 
\leq 
\leb(I)\leb(J) + \frac{2}{2^{[\frac{n}{2}]}} \leb(J).
\]
\end{lemma}
\begin{proof}
We may further suppose, without losing any generality, that $I$ and $J$ are left open and right closed intervals. Write $I=(\alpha,\beta]$. Assume that $\alpha>0$. Let $\scI_k^{\alpha}$ and $\scI_k^{\beta}$ be the fundamental intervals of order $k\colon= [\frac{n}{2}]$ containing $\alpha$ and $\beta$, respectively. Put
\[
I_0:=I\cap \scI_k^{\alpha}, \quad I_2:=I\cap \scI_k^{\beta}, \quad I_1:=I\setminus(I_0\cup I_1).
\]
If $\scI_k^{\alpha}=\scI_k^{\beta}$, then $I\subseteq \scI_k^{\alpha}$ and, by Lemma \ref{Le:Dist-A},
\begin{align*}
\leb\left(I\cap\scL^{-n}[J]\right) &\leq \leb\left(\scI_k^{\alpha}\cap \scL^{-n} [J]\right) \nonumber\\
&= \leb(\scI_k^{\alpha})\leb(J) \nonumber\\
&\leq \frac{1}{2^k}\leb(J)\leq \leb(I)\leb(J) + \frac{2}{2^k} \leb(J).
\end{align*}

Assume now that $\scI_k^{\alpha}\neq \scI_k^{\beta}$. Since $I_1$ is a union of fundamental intervals of order $k$, we may use Lemma \ref{Le:Dist-A} to obtain $\leb(I_1\cap \scL^{-n}[J])=\leb(I_1)\leb(J)$. Hence, 
\begin{align*}
\leb(I\cap \scL^{-n}[J]) &= \leb\left(I_0\cap \scL^{-n}[J]\right) + \leb\left(I_1\cap \scL^{-n}[J]\right) + \leb\left(I_2\cap \scL^{-n}[J]\right)
 \nonumber\\
&\leq \leb\left(\scI_k^{\alpha}\cap \scL^{-n}[J]\right) + \leb(I_1)\leb(J) + \leb\left(\scI_k^{\beta}\cap \scL^{-n}[J]\right) \nonumber\\
&=\leb(I_1)\leb(J) + \left( \leb\left(\scI_k^{\alpha}\right)+\leb\left( \scI_k^{\beta}\right)\right) \leb(J) \nonumber\\
&\leq \leb(I)\leb(J) + \frac{2}{2^k} \leb(J).
\end{align*}
This shows the lemma when $\alpha>0$.

If $\alpha=0$, consider a strictly decreasing sequence $(\alpha_m)_{m\geq 1}$ of real numbers contained in $I=(0,\beta]$ converging to $0$. For each $m\in\Na$, the previous argument shows that
\[
\leb\left( [\alpha_m,\beta) \cap \scL^{-n}[J]\right) 
\leq 
\leb([\alpha_m,\beta))\leb(J) + \frac{2}{2^{[\frac{n}{2}]}} \leb(J).
\]
The result follows by considering $m\to\infty$.
\end{proof}
\begin{proof}[Proof of Proposition \ref{Teo:PhilippLuroth}]
Define the sequence $(E_n)_{n\geq 1}$ by $E_n=\scL^{-n}[I_n]$ for all $n\in\N$. Therefore, when $m,n\in \N$ satisfy $m<n$, Lemma \ref{Le:Dist-A} and the $\scL$-invariance of $\leb$ give
\begin{align*}
\leb(E_n\cap E_m) &= \leb\left( \scL^{-n}[I_n]\cap \scL^{-m}[I_m]\right)\\
&= \leb\left( \scL^{-(n-m)}[I_n]\cap I_m \right) \\
&\leq \leb(I_n)\leb(I_m) + \frac{2}{2^{[\frac{n-m}{2}]} }\leb(I_m) \\
& = \leb(E_n)\leb(E_m) + \frac{2}{2^{[\frac{n-m}{2}]} }\leb(E_m).
\end{align*}
The result follows now from Theorem \ref{Teo:PhilipBC}, because $\sum_n \frac{2}{2^{[n/2]}}$ converges.
\end{proof}
For any sequence of positive real numbers $\bfr=(r_n)_{n\geq 1}$ define
\[
E(\bfr)=\{ y\in [0,1]: |\Lu^n(x)-\Lu^n(y)|<r_n \text{ for infinitely many } n\in\Na \}.
\]
Then, we have 
\[
\leb(E(\bfr)) = 
\begin{cases}
0, \text{ if } \sum_n r_n<+\infty, \\
1, \text{ if } \sum_n r_n=+\infty.
\end{cases}
\]
The convergence part follows from the Borel-Cantelli lemma and the divergence part, from Proposition \ref{Teo:PhilippLuroth}. In particular, if $r_n=n^{-1}$ for all $n\in\N$, for almost every $y\in [0,1]$ there exists an increasing sequence of integers $(n_j)_{j\geq 1}$ such that $|\scL^{n_j}(x)-\scL^{n_j}(y)|<r_j$ holds for all $j\in\Na$. In other words, almost every $y\in [0,1]$ satisfies
\[
\liminf_{n\to\infty} \left|\Lu^{n}(x)- \Lu^{n}(y)\right|=0.
\]
Hence, by the definition of distal pairs, $\leb(\noD_{\scL}(x))=0$.
\subsection{The distal pairs of $x$ have full dimension}
The proof consists of two parts. First, we construct a subset of $\noD_{\scL}(x)$. Then, we estimate its Hausdorff dimension.

\subsubsection{A collection of distal pairs}

Given a collection $\bfe=(\bfe^{m,n})_{m,n\geq 2}$ of sequences on $\scD$ indexed on $\Na_{\geq 2}\times \Na_{\geq 2}$, we define for each $m,n\in\Na$ the set
\[
F_m^n(\bfe):=\left\{\scB_m(\bfb,\bfe^{m,n}):\bfb\in \intent{2,n}^{\N}\right\}
\]
(see \cref{sec:symbolic}).
\begin{lemma}\label{Le:DisC}
Let $\bfe$ be as above. For any $m_0,n_0\in\N_{\geq 2}$, the following set is dense in $\scD^{\Na}$:
\[
\bigcup_{\substack{m\geq m_0 \\ n\geq n_0}} F_m^n(\bfe).
\]
Moreover, the set
\[
\bigcup_{\substack{m\geq m_0 \\ n\geq n_0}} \Lambda\left[ F_m^n(\bfe)\right].
\]
is dense in $(0,1]$.
\end{lemma}
\begin{proof}
The first assertion follows from the definition of the product topology. The second assertion follows from the first one and Proposition \ref{Teo:LuLambda}.
\end{proof}
The next lemma tells us how to choose an appropriate $\bfe$.
\begin{lemma}\label{Le:LuDisD}
Let $E\in \N_{\geq 6}$, $m\in\N_{\geq 2}$, and $\bfc=(c_1,\ldots,c_m)\in\clD^m$. If $\bfb=(b_1,\ldots,b_{m-1},e)\in \intent{2,E}^m$ satisfies
\[
3\leq e\leq E-1 \;\text{ and }\; 2\leq |c_m-e|,
\]
then $d\left( \scI_m(\bfc),\scI_m(\bfb) \right) \geq E^{-2m}$.
\end{lemma}
\begin{proof}
The restriction on $E$ is imposed to ensure the existence of some $\bfb$ satisfying the hypotheses. The second condition on $e$ implies that the intervals $\scI_m(\bfc)$ and $\scI_m(\bfb)$ are not adjacent, hence
\[
d\left(\mathcal{I}_{m}(\bfc), \mathcal{I}_{m}(\bfb)\right) 
\geq |\mathcal{I}_{m}(b_1, \ldots, b_{m-1}, e+1)|
= \frac{1}{e(e+1)} \prod_{k=1}^{m-1} \frac{1}{b_k(b_k-1)} \nonumber\\
\geq \frac{1}{E^{2m}}.
\]
\end{proof}
For the rest of the proof, $E\in\N_{\geq 6}$ will remain fixed. For each $n\in\N_{\geq E}$ and each $m\in\N_{\geq 2}$, let $\bfe^{m,n}=\bfe^m = (e^m_j)_{j\geq 1}\in\intent{3,E-1}^\N$ be a sequence on $\scD$ such that $|a_{jm} - e_j^m|\geq 2$ for all $j\in\N$ and write $\bfe=(\bfe^{m,n})_{m,n}$. By Lemma \ref{Le:LuDisD}, for any $m\geq 2$ and any $\bfb=\sabu\in F_m^{n}(\bfe)$, the number $y=\Lambda(\bfb)\in(0,1]$ satisfies $|y-x|> E^{-2m}$. In fact, when $j\in \intent{0,m-1}$, we have $|\Lu^j(x) - \Lu^j(y)|\geq E^{-2(m-j)}$, so
\begin{equation}\label{Eq-Cota-Distal}
|\Lu^j(x) - \Lu^j(y)|\geq \frac{1}{E^{2(m-j)}} \geq \frac{1}{E^{2m}}.
\end{equation}
Since $\scL^m(x)=\left\langle a_{m+1}, a_{m+2},\ldots\right\rangle$ and $\scL^m(y)=\left\langle b_{m+1}, \ldots,b_{2(m-1)-1},e_2^m,\ldots \right\rangle$, inequality \eqref{Eq-Cota-Distal} holds for $j=m$. Moreover, an inductive argument gives \eqref{Eq-Cota-Distal} for all $j\in\Na_0$. As a consequence, $\Lambda[F_m^n(x)]\subseteq \noD_{\scL}(x)$ and, since $m$ and $n$ were arbitrary,
\begin{equation}\label{Ec:DL01}
\bigcup_{\substack{m\geq 2 \\ n\geq E}}\Lambda\left[ F_m^n(\bfe)\right]\subseteq \noD_{\scL}(x).
\end{equation}
\subsubsection{Hausdorff dimension estimate}
We will use Lemma \ref{Le:GJL01} to show that the smaller set in \eqref{Ec:DL01} has full Hausdorff dimension. To be more precise, for any $s\in (0,1)$ we find some $m,n\in\Na$ such that $\dimh \Lambda(F_m^n(\bfe))\geq s$ which, due to \eqref{Ec:DL01}, implies $\dimh \noD_{\scL}(x)=1$.

Given $n,m\in\N_{\geq 2}$, write $\scA_0=\{[0,1]\}$ and for every $k\in\N$ and $\bfc=(c_1,\ldots,c_{k(m-1)})\in\intent{2,n}^{k(m-1)}$ define 
\[
A^k(\bfc) :=
A^k(\bfc;n,m) :=  \overline{\clI_{mk}}\left( \scB_m \left( \bfc,(e_1^m,e_2^m,\ldots, e_k^m)\right)\right)
\]
and 
\[
\scA_k\colon=\left\{ A^k(\bfc) \colon \bfc\in \intent{2,n}^{k(m-1)}\right\}.
\]
The limit set of $\scA(m,n)=\{\scA_k:k\in\Na_0\}$ is precisely $\Lambda[F_m^n(\bfe)]$.
\begin{proposition}\label{PROPO-HJGL-01}
For each $m\in\N_{\geq 2}$ and $n\in\N_{\geq E}$, there is a constant $B=B(m,n)>0$ such that every $k\in\N$, $\bfc\in\intent{2,n}^{k(m-1)}$ and $\bff,\bfb\in\intent{2,n}^{m-1}$ with $\bfb\neq \bff$ satisfy
\[
d\left( A^{k+1}(\bfc\bfb),A^{k+1}(\bfc\bff)\right) \geq B |A^k(\bfc)|.
\]
\end{proposition}
\begin{proof}
Let $m\in\N_{\geq 2}$ and $n\in \N_{\geq E}$ be arbitrary. For any $e\in\intent{3,E-1}$, the compact sets belonging to $\{\overline{\clI}_m(\bfb e): \bfb\in \intent{2,n}^{m-1}\}$ are pairwise disjoint. Hence, there is some constant $B_1(n,m;e)>0$ such that for two different $\bfb,\bff\in\intent{2,n}^{m-1}$ we have
\[
d\left(\overline{\clI}_m(\bfb e),\overline{\clI}_m(\bff e)\right) \geq B_1(m,n;e).
\]
Put
\[
B:=B(m,n):=\min\{B_1(m,n;3), \ldots, B_1(m,n;E-1)\}>0. 
\]
For any $k\in\N$, $\bfc\in\intent{2,n}^{(m-1)k}$, and $\bfb,\bff\in\intent{2,n}^{m-1}$ with $\bfb\neq\bff$, we write $\mathbf{g}:=\scB_m(\bfc,\bfe^m)$. Then, we obtain
\begin{align*}
d\left( A^{k+1}(\bfc\bfb),A^{k+1}(\bfc\bff)\right) &= d\left( \scT_{\bfg}^{km}[A^1(\bfb)],\scT_{\bfg}^{km}[A^1(\bff)]\right) \nonumber\\
&=|A^k(\bfc)| \;d\left( A^1(\bfb),A^1(\bff)\right) \nonumber\\
&\geq |A^k(\bfc)| B. \nonumber
\end{align*}
\end{proof}
	
\begin{proposition}\label{PROPO-HJGL-02}
For each $m\in\N_{\geq 2}$, $n\in\N_{\geq E}$, and $k\in\Na$ define
\[
d_k\left( \clA(m,n)\right) := \max\left\{ |A^k(\bfc)|: \bfc\in \intent{2,n}^{k(m-1)}\right\}.
\]
Then, we have
\[
\liminf_{k\to\infty} \frac{\log \left( d_k(\clA(m,n))^{-1}  \right)}{k}>0.
\]
\end{proposition}
\begin{proof}
Let $m,n,k$ be as in the statement and $\bfc\in\intent{2,n}^{k(m-1)}$. Then, $|A^k(\bfc)|\leq 2^{-mk}$ and
\[
d_k\left(\clA(m,n)\right)\leq \frac{1}{2^{km}}.
\]
Direct computations lead to the desired conclusion.
\end{proof}
	
\begin{proposition}\label{PROPO-05}
For each $s\in (0,1)$, there are $m\in\N_{\geq 2}$, $n\in\N_{\geq E}$ such that
\[
\dimh \Lambda[F_m^n(\bfe)]\geq s.
\]
\end{proposition}
\begin{proof}
Take $s\in (0,1)$. Let $m\in\N_{\geq 2}$ and $n\in\N_{\geq E}$ be such that
\[
n\geq E, \quad 
\sum_{j=2}^n \frac{1}{j^s(j-1)^s} >1, \;\text{ and }\; 
\frac{1}{E^2} \left( \sum_{j=2}^n \frac{1}{j^s(j-1)^s} \right)^{m-1}>1.
\]
Let $k\in\N$ and $\bfc\in\intent{2,n}^{k(m-1)}$ be arbitrary. For every $\bfb=(b_1,\ldots, b_{m-1})\in\intent{2,n}^{m-1}$ we have
\[
\left|A^{k+1}(\bfc \bfb) \right| 
= \left| A^k(\bfc)\right| \frac{1}{e_{k+1}^m(e_{k+1}^m-1)} \prod_{j=1}^{m-1} \frac{1}{b_j (b_j-1)} 
\geq |A^k(\bfc)| \frac{1}{E^2}\prod_{j=1}^{m-1} \frac{1}{b_j (b_j-1)}.
\]
Therefore, letting $\bfb$ run along $\intent{2,n}^{m-1}$, we obtain
\begin{align*}
\sum_{\bfb} |A^{k+1}(\bfc\bfb)|^s &\geq |A^k(\bfc)|^s \frac{1}{E^{2s}}\sum_{\bfb} \prod_{j=1}^{m-1} \frac{1}{b_j^s (b_j-1)^s} \\
&= |A^k(\bfc)|^s \frac{1}{E^{2s}}\left(\sum_{j=2}^n \frac{1}{j^s (j-1)^s}\right)^{m-1} \\
&\geq |A^k(\bfc)|^s \frac{1}{E^2} \left( \sum_{j=2}^n \frac{1}{j^s(j-1)^s} \right)^{m-1} > |A^k(\bfc)|^s. \nonumber
\end{align*}
Let $B=B(m,n)$ be as in Proposition \ref{PROPO-HJGL-01}. Taking into account the constant sequence $(B_n)_{n\geq 1}$ with $B_n=B$ and Proposition \ref{PROPO-HJGL-02}, we can apply Lemma \ref{Le:GJL01} to conclude
\[
\dimh \Lambda\left[ F_m^n(\bfe)\right] \geq s. 
\]
\end{proof}
We are now in the position to estimate the Hausdorff dimension of $\noD_{\scL}(x)$. In view of \eqref{Ec:DL01} and Proposition \ref{PROPO-05}, we have
\[
\dimh \left(\bigcup_{\substack{m\geq 2 \\ n\geq E}}\Lambda\left[ F_m^n(\bfe)\right]\right) 
= \sup_{\substack{m\geq 2 \\ n\geq E}} \dimh \left(\Lambda\left[ F_m^n(\bfe)\right]\right) 
=1,
\]
so $\dimh \noD_{\scL}(x)=1$. The density of $\noD_{\scL}(x)$ follows immediately from Proposition \ref{Le:DisC} and the proof of Theorem \ref{TEO-DISTAL-LUROTH} is now complete.

\section{Devaney chaos: proof of Theorem \ref{TEOREMA-DEVCHAOS-LUROTH}}
\label{devaney}

\noindent Let us show $([0,1],\scL)$ is topologically transitive. Let $U,V$ be two non-empty open subsets of $[0,1]$. Pick a non-zero $x\in V$ and write $x=\langle a_1,a_2,a_3,\ldots\rangle$. Take a non-zero $y\in U$, write $y=\langle b_1,b_2,b_3,\ldots \rangle$, and choose $n\in\Na$ such that $\scI_n(\bfb)\subseteq U$ (it exists by Corollary \ref{Co:Luroth-02}). Then, the point 
\[
z=\langle b_1,b_2,\ldots,b_n,a_1,a_2,\ldots\rangle
\]
belongs to $\scI_n(\bfb)\subseteq U$ and $\scL^n(z)=x\in V$; that is, $\scL^n(z)\in \scL^n[U]\cap V$. Using similar ideas, it is not hard to show that $([0,1],\scL)$ is locally eventually onto; that is, for every non-empty open subset $U$ there exists some $n\in\Na$ such that $\scL^n[U]=[0,1]$.

Since $\Lambda\colon \scD^{\Na}\to [0,1]$ is continuous and onto, it maps dense subsets of $\scD^{\Na}$ onto dense subsets of $[0,1]$. In particular, $\Lambda$ maps the set of periodic sequences to a dense subset of $[0,1]$. Moreover, by $\Lambda \circ \sigma = \scL\circ \Lambda$, the function $\Lambda$ maps periodic sequences onto periodic points with respect to $\scL$ (see Proposition \ref{Teo:LuLambda}).

Finally, we show that $([0,1],\scL)$ is sensitive to initial conditions. Let $\delta>0$ be arbitrary and take $x\in [0,1]$. For $x=0$, take $N\in\Na_{\geq 2}$ such that $N^{-1}<\delta$. Then, the point
\[
y = \langle N+1,2,2,2,2,2,\ldots\rangle  = \frac{1}{N}
\]
satisfies $|x-y|<N^{-1}< \delta$ and $|\scL(x)-\scL(y)|=|0-1|=1$. If $x\in (0,1]$, write $x=\langle a_1,a_2,a_3,\ldots\rangle$. Take $n\in\Na$ such that $\scI_n(\bfa)\subseteq B(x;\delta)$ and let $z$ be the only element in the set
\[
\left\{ \scL^n(x) - \frac{1}{2}, \scL^n(x) + \frac{1}{2}\right\} \cap (0,1].
\]
Write $z=\langle c_1,c_2,c_3,\ldots\rangle$ and define 
\[
y=\langle a_1,a_2,\ldots,a_n,c_1,c_2,c_3,\ldots\rangle.
\]
Hence, we have $|x-y|<\delta$ and $|\scL^n(x) - \scL^n(y)|=\frac{1}{2}$. Therefore, the system $([0,1],\scL)$ is Devaney chaotic. 

It is not hard to adapt the argument above to show that $([0,1),\scG)$ is also Devaney chaotic.

\section{A scrambled set: proof of Theorem \ref{TEOREMA-SCRAM-LUROTH-01}}
\label{liyorke}

\noindent The proof is divided into three parts. In the first part, we construct an $\scL$-scrambled set $S$. Afterwards, we show that $S$ is in fact scrambled with respect to $\scL$. Finally, we show that $S$ has Hausdorff dimension $1$.
\subsection{Construction of $S$}
For each $\bfa=\sanu \in\clD^{\N}$, let $g(\bfa)\in\clD^{\N}$ be the limit of the finite words $(\bfA_n)_{n\geq 1}$ given by $\bfA_1=2a_1$ and $\bfA_{n+1}=\bfA_n 2^{n+1}\pref_{n+1}(\bfa)$; that is, 
\[
g(\bfa) = (2,a_1,2,2,a_1,a_2, 2,2,2,a_1,a_2,a_3,\ldots).
\]
Consider $R=\bigcup_{m\in\N} \intent{m^3+1,m^3+2m}$ and define $\clF:\clD^{\N}\to\clD^{\N}$ by
\[
\clF(\bfa)=\scB(\bfa,g(\bfa);R) 
\text{ for all } \bfa\in \clD^{\N}.
\]
Note that both functions $g$ and $\clF$ are continuous. For each $N\in\N_{\geq 3}$ define
\[
\widetilde{S}_N\colon = \clF\left[ \intent{2,N}^{\N}\right]\subseteq \scD^{\Na}, \quad
S_N \colon= \Lambda\left[ \widetilde{S}_N\right]\subseteq (0,1]\]
and put $S\colon=\bigcup_{N\geq 3} S_N$. 
\subsection{The set $S$ is $\scL$-scrambled}
We now show that $S$ is scrambled with respect to $\scL$. In other words, if $\alpha,\beta\in S$ are different, then
\[
\liminf_{n\to\infty} |\scL^n(\alpha) - \scL^n(\beta)|=0 \;\text{ and }\; 
\limsup_{n\to\infty} |\scL^n(\alpha) - \scL^n(\beta)|>0.
\]
Take $N\in\Na_{\geq 3}$. Let $\bfa=\sanu, \bfb=\sabu\in \intent{2,N}^\N$ be such that $\bfa\neq\bfb$. Define
\[
x \colon=\Lambda(\bfa),  \quad
\alpha \colon= \Lambda\circ \clF(\bfa), \quad
y \colon=\Lambda(\bfb), \quad 
\beta \colon= \Lambda\circ \clF(\bfb). 
\]
For each $m\in\N$ we have
\[
\Lu^{m^3}(\alpha) = \langle \underbrace{2,2,\ldots, 2}_{m \text{ times} },a_1,\ldots, a_m,\ldots\rangle \;\text{ and }\;
\Lu^{m^3}(\beta) = \langle \underbrace{2,2,\ldots, 2}_{m \text{ times} },b_1,\ldots, b_m,\ldots\rangle;
\]
hence,
\[
\lim_{m\to\infty}  \Lu^{m^3}(\alpha) = \lim_{m\to\infty}  \Lu^{m^3}(\beta) =1
\]
and
\[
\liminf_{n\to\infty} \left| \Lu^n(\alpha) - \Lu^n(\beta)\right| = 0.
\]
On the other hand, since
\[
\scL^{m^3+m}(\alpha) = \langle a_1,\ldots, a_m, 2, \ldots\rangle \;\text{ and } 
\scL^{m^3+m}(\beta) = \langle b_1,\ldots, b_m,2\ldots\rangle, \nonumber
\]
we have
\[
\lim_{m\to\infty} \scL^{m^3+m}(\alpha) = x \;\text{ and }\;
\lim_{m\to\infty} \scL^{m^3+m}(\beta) = y.
\]
Because $\Lambda$ is bijective, we have $x\neq y$ and 
\[
\limsup_{n\to\infty} \left| \Lu^n(\alpha) - \Lu^n(\beta)\right| 
\geq |x - y|>0.
\]
Therefore, $(\alpha,\beta)$ is an $\scL$-scrambled pair and $S$ is an $\scL$-scrambled set.

\subsection{Hausdorff dimension estimates}
Define $Q:=\N\setminus R$ and write $Q=\{q_n:n\in\Na\}$ with $q_n<q_{n+1}$ for all $n\in\Na$. Let $\pi_{Q}$ be the function
\[
\pi_Q\colon\bigcup_{N\geq 3} \widetilde{S}_N\to \bigcup_{N\geq 3} \intent{2,N}^{\N}
\]
given by
\[
\pi_Q\left(\sanu\right) = (a_{q_j})_{j\geq 1}.
\]
Recall that $F_N\colon=\Lambda[\intent{2,N}]$ for each $N\in\Na_{\geq 3}$. The map $\scY\colon=\clF^{-1}$,
\[
\scY\colon \bigcup_{N\geq 3} \widetilde{S}_N\to \bigcup_{N\geq 3} F_N,
\]
is given by $\scY = \Lambda\circ\pi_Q\circ \Lambda^{-1}$ and is continuous. Moreover, for each $N\in\Na_{\geq 3}$, the restriction of $\scY$ to $S_N$ establishes a homeomorphism between $S_N$ and $F_N$ which is, by Lemma \ref{Lem-Rev-01} below, locally Hölder-continuous.

\begin{lemma}\label{Lem-Rev-01}
Given $N\in\N_{\geq 3}$ and $\veps>0$, there exists $r=r(N,\veps)>0$ such that for any $\beta\in S_N$ and any $\alpha\in S_N\cap B(\beta;r)$
\[
\left| \scY(\alpha) - \scY(\beta) \right| \ll_N |\alpha - \beta|^{\frac{1}{1+\veps}}.
\]
\end{lemma}
We need to two preliminary results for the proof of Lemma \ref{Lem-Rev-01}.
\begin{proposition}\label{Lem-Rev-02}
If $t(n)\colon=\#(R\cap\intent{1,n})$ for every $n\in\N$, then $t(n)\asymp n^{\frac{2}{3}}$.
\end{proposition}
\begin{proof}
Take $n\in\N_{\geq 8}$ and define $m=[ n^{\frac{1}{3}}]$. Then, $2\leq m\leq n^{\frac{1}{3}}< m+1$, so
\begin{align*}
t(n) &\leq 2\sum_{j=1}^m j = m(m+1) = m^2\left( 1+ \frac{1}{m}\right) \leq 2n^{\frac{2}{3}}, \text{ and } \nonumber\\
t(n) &\geq 2\sum_{j=1}^{m-1} j = (m-1)m = (m+1)^2 \left( 1- \frac{1}{m+1}\right)\left( 1- \frac{2}{m+1}\right) \geq \frac{2}{9} n^{\frac{2}{3}}. \nonumber
\end{align*}
\end{proof}

\begin{proposition}\label{Lem-Rev-03}
Let $\gamma=\left\langle c_1,c_2,c_3,\ldots\right\rangle$, $\delta=\left\langle h_1,h_2,h_3,\ldots\right\rangle$ belong to $F_N$ for some $N\in\N_{\geq 3}$ (see \eqref{Eq:DefFN}). If $n_0\in\N$ is such that $c_j=h_j$ for each $j\in\intent{1,n_0}$ and $c_{n_0+1}\neq h_{n_0+1}$, then 
\[
\left| \gamma - \delta\right| 
\geq \left|\scI_{n_0}(c_1,\ldots,c_{n_0}) \right| \frac{1}{N^3}.
\]
\end{proposition}
\begin{proof}
Without loss of generality, assume that $c_{n_0+1}>h_{n_0+1}$. From this inequality,
\[
\scL^{n_0}(\gamma) \leq \frac{1}{c_{n_0+1}-1} \;\text{, and }\;
\frac{1}{h_{n_0+1}} + \frac{1}{h_{n_0+1}(h_{n_0+1}-1) N} < \scL^{n_0}(\delta),
\]
we obtain
\[
\left| \scL^{n_0} (\gamma) -\scL^{n_0} (\delta)\right| 
=\scL^{n_0}(\delta) - \scL^{n_0}(\gamma)
\geq \frac{1}{N^2(N-1)}> \frac{1}{N^3}.
\]
Writing $\bfc=(c_j)_{j\geq 1}$, we conclude
\begin{align*}
|\gamma - \delta| &= \left| \scT^{n_0}_{\bfc}\left( \scL^{n_0}(\gamma)\right) - \scT^{n_0}_{\bfc}\left( \scL^{n_0}(\delta)\right)\right| \\
&= \left| \scI_{n_0}(\bfc)\right| \left| \scL^{n_0}(\gamma) - \scL^{n_0}(\delta)\right| \geq \left| \scI_{n_0}(\bfc)\right| \frac{1}{N^3}.
\end{align*}
\end{proof}
\begin{proof}[Proof of Lemma \ref{Lem-Rev-01}]
Assume that $\veps>0$ and $N\in\N_{\geq 3}$ are given and let $n_0\in\N$ satisfy
\begin{equation}\label{Eq:DemLem61-01}
\frac{n}{t(n)} \geq 1+ \frac{2\log N}{\veps \log 2} \text{ for all } n\in\N_{\geq n_0}
\end{equation}
(it exists by Proposition \ref{Lem-Rev-02}). Take $\bfb\in\widetilde{S}_N$ and put $\beta=\Lambda(\bfb)$. Let $r>0$ be such that the first $n_0$ Lüroth digits of every number in $S_N\cap B(\beta;r)$ coincide with those of $\beta$. It is not hard to show that such an $r$ actually exists since $S_N$ can be expressed as a countable intersection of a nested sequence compact sets, each of which is defined by restricting the first L\"uroth digits. Furthermore, $r$ can be chosen independently of $\beta$ by Corollary \ref{Co:Luroth-02}. Let $\gamma=\left\langle c_1,c_2,c_3,\ldots\right\rangle\neq \beta$ belong to $S_N\cap B(\beta;r)$, write $\bfc=(c_n)_{n\geq 1}$, and call $n\in\N$ the largest natural number satisfying $\pref(\bfc;n)=\pref(\bfb;n)$. Then, we have
\begin{align*}
|\scI_n(b_1,\ldots, b_n)| &= \prod_{j=1}^n \frac{1}{b_j(b_j-1)} \nonumber\\
&= \left( \prod_{j\in Q\cap\intent{1,n}} \frac{1}{b_j(b_j-1)} \right)\left( \prod_{j\in R\cap\intent{1,n}} \frac{1}{b_j(b_j-1)} \right) \nonumber\\
&\geq \left( \prod_{j\in Q\cap\intent{1,n}} \frac{1}{b_j(b_j-1)} \right) \frac{1}{N^{2t(n)}}\nonumber\\
&= |\scI_{n-t(n)} (\pi_Q(\bfb))| \frac{1}{N^{2t(n)}} \geq \left|\scI_{n-t(n)} (\pi_Q(\bfb))\right|^{1+\veps}.
\end{align*}
The last inequality holds by \eqref{Eq:DemLem61-01} and $|\scI_{n-t(n)}(\pi_Q(\bfb))|\leq 2^{-(n-t(n))}$. Hence, in view of Proposition \ref{Lem-Rev-03}, we conclude that
\[
\left| \scY_N(\gamma) - \scY_N(\delta)\right| 
\leq \left| \scI_{n-t(n)} \left( \pi_Q(\bfb)\right) \right| 
\leq \left| \scI_{n} (\bfb) \right|^{\frac{1}{1+\veps}} 
\leq \left| \scI_{n_0} (\bfb) \right|^{\frac{1}{1+\veps}} 
\ll_N |\gamma - \delta|^{\frac{1}{1+\veps}}.
\]
\end{proof}

\subsection{Proof of Theorem \ref{TEOREMA-SCRAM-LUROTH-01}}
Let $N\in\Na_{\geq 3}$ be arbitrary. Given $\veps>0$, take $r>0$ as in Lemma \ref{Lem-Rev-01} and let $x_1,\ldots,x_n$ be points in the compact set $S_N$ such that
\begin{equation}\label{Eq:TeoSL01}
S_N\subseteq \bigcup_{j=1}^n B(x_j;r).
\end{equation}
Lemma \ref{Lem-Rev-01} guarantees that, for each $j\in\intent{1,n}$, the function $\scY$ restricted to $B(x_0;r)\cap S_N$ is $(1+\veps)^{-1}$-Hölder continuous, so
\[
(1+\veps)^{-1} \dimh \left(\scY\left[ B(x_j;r)\cap S_N\right]\right) 
\leq 
\dimh \left( B(x_j;r)\cap S_N \right)
\]
(see Proposition 3.3 in \cite{Fal2014}). Then, by \eqref{Eq:TeoSL01} and $\scY[S_N]=F_N$, 
\begin{align*}
\dimh S_N &= \dimh \left(\bigcup_{j=1}^n B(x_j;r)\cap S_N \right) \\ 
&= \max_{1\leq j\leq n} \dimh \left(B(x_j;r)\cap S_N \right) \\
&\geq (1+\veps)^{-1} \max_{1\leq j\leq n} \dimh \left(\scY\left[B(x_j;r)\cap S_N\right]\right) \\
&= (1+\veps)^{-1} \dimh\left( \bigcup_{j=1}^n \scY\left[ B(x_0;r)\cap S_N\right]\right) \\ 
&=(1+\veps)^{-1} \dimh\left( \scY\left[S_N\right] \right)= (1+\veps)\dimh (F_N).
\end{align*}
Since $\veps>0$ was arbitrary and $S_N\subseteq F_N$, we have $\dimh (S_N)=\dimh (F_N)$ and Theorem \ref{TEOREMA-SCRAM-LUROTH-01} follows:
\[
\dimh (S) 
= \sup_{N} \dimh (S_N) 
= \lim_{N\to\infty} \dimh (S_N)
= \lim_{N\to\infty} \dimh (F_N) = 1. 
\]

\section{Further problems}
\label{problems}

\noindent Theorem \ref{TEO-ASYM-LUROTH} provides uncountably many numbers $x\in [0,1]$ such that $\dimh \AS_{\scL}(x)=0$ and uncountably many $x\in [0,1]$ such that $\dimh \AS_{\scL}(x)=\frac{1}{2}$. It also shows that $\dimh \AS_{\scL}(x)$ lies between $0$ and $\frac{1}{2}$ for any $x\in [0,1]$. However, it does not say what happens between these values. These observations prompt new questions.

\begin{problem}
What is the image of the function $x\mapsto \dimh \AS_{\scL}(x)$?
\end{problem}

Define
\begin{align*}
\Theta(\scL;0) &\colon= \left\{ x\in [0,1]\colon \dimh \textstyle\AS_{\scL}(x)=0 \displaystyle\right\}, \\
\Theta\left(\scL;\frac{1}{2}\right) &\colon= \left\{ x\in [0,1]\colon \dimh \textstyle\AS_{\scL}(x)\displaystyle=\frac{1}{2}\right\}.
\end{align*}
By Theorem \ref{TEO-ASYM-LUROTH}, $\Theta(\scL;0)$ contains the set of numbers in $[0,1]$ whose Lüroth expansion has bounded digits, hence 
\[
\dimh \Theta(\scL;0) = 1.
\]
However, Theorem \ref{TEO-ASYM-LUROTH} only allows us to conclude that $\dimh \Theta(\scL,\frac{1}{2})\geq \frac{1}{2}$, for it contains all the real numbers in $[0,1]$ whose Lüroth digits tend to infinity.
\begin{problem}
What is the Hausdorff dimension of $\Theta(\scL,\frac{1}{2})$?
\end{problem}

We can refine the previous problems if $x\mapsto \dimh \AS_{\scL}(x)$ takes more than two values.
\begin{problem}
Estimate for each $\alpha\in [0,\frac{1}{2}]$ the Hausdorff dimension of the set
\[
\Theta(\scL;\alpha) \colon= \{x\in [0,1]\colon \dimh \textstyle\AS_{\scL} \displaystyle(x)=\alpha\}.
\]
\end{problem}

Recall that Theorems \ref{TEO-ASYM-LUROTH}, \ref{TEO-DISTAL-LUROTH}, and \ref{TEOREMA-SCRAM-LUROTH-01} are full analogues of the main results in \cite{LiuLi2017}, \cite{LiuWan2019}. It is thus natural to pose problems above for $([0,1),\scG)$.

\begin{problem}
What is the image of the function $x\mapsto \dimh\AS_{\scG}(x)$?
\end{problem}
Consider the sets
\begin{align*}
\Theta(\scG;0) &\colon= \left\{ x\in [0,1)\colon \dimh \textstyle\AS_{\scG}\displaystyle(x)=0\right\}, \\
\Theta\left(\scG;\frac{1}{2}\right) &\colon= \left\{ x\in [0,1)\colon \dimh \textstyle\AS_{\scG}\displaystyle(x)=\frac{1}{2}\right\}.
\end{align*}
Since the set of badly approximable numbers-- that is, those real numbers whose regular continued fraction is bounded-- has full Hausdorff dimension (see \ref{Eq:IntroJarnik}), Theorem \ref{thm:sumarised} implies that $\dimh \Theta(\scG;0) =1$. On the other hand, by Good's Theorem \eqref{Eq:IntroGood} and Theorem \ref{thm:sumarised}, we know that $\dimh \Theta(\scG;\frac{1}{2}) \geq  \frac{1}{2}$. 
\begin{problem}
What is the Hausdorff dimension of $\Theta(\scG;\frac{1}{2})$?
\end{problem}
\begin{problem}
For each each $\alpha\in [0,\frac{1}{2}]$, estimate the Hausdorff dimension of the set
\[
\Theta(\scG;\alpha) \colon= \{x\in [0,1)\colon \dimh \textstyle\AS_{\scG}\displaystyle(x) =\alpha\}.
\]
\end{problem}

\noindent 

\section*{Acknowledgements}
\noindent Research by R. Alcaraz Barrera was funded by CONACYT-M\'exico Award 740785 through the program ``Programa de retenci\'on y repatriaci\'on''. Research by G. González Robert was partially funded by CONACYT-Mexico through the program ``Estancias posdoctorales por México. Modalidad 1: Estancia posdoctoral académica''. The paper was finished during R. Alcaraz Barrera's visit to the National Autonomous University of Mexico during October 2021. The visit was sponsored by DGAPA-PAPIIT no. 110221 ``Sistemas Din\'amicos Simb\'olicos y Combinatoria Anal\'itica, II''. We are grateful to Ricardo G\'omez Aiza for his hospitality. We thank the referees for their comments.

\bibliography{luroth}

\begin{thebibliography}{10}

\bibitem{ArrGero2021}
A.~Arroyo and G.~Gonz\'{a}lez~Robert.
\newblock Hausdorff dimension of sets of numbers with large {L}\"{u}roth
  elements.
\newblock {\em Integers}, 21:Paper No. A71, 20, 2021.

\bibitem{BarVal2013}
L.~Barreira and C.~Valls.
\newblock {\em Dynamical systems}.
\newblock Universitext. Springer, London, 2013.
\newblock An introduction, Translated from the 2012 Portuguese original.

\bibitem{BisPer2017}
C.~J. Bishop and Y.~Peres.
\newblock {\em Fractals in probability and analysis}, volume 162 of {\em
  Cambridge Studies in Advanced Mathematics}.
\newblock Cambridge University Press, Cambridge, 2017.

\bibitem{BlaGlaKolMaa2002}
F.~Blanchard, E.~Glasner, S.~Kolyada, and A.~Maass.
\newblock On {L}i-{Y}orke pairs.
\newblock {\em J. Reine Angew. Math.}, 547:51--68, 2002.

\bibitem{DajKra1996}
K.~Dajani and C.~Kraaikamp.
\newblock On approximation by {L}\"{u}roth series.
\newblock {\em J. Th\'{e}or. Nombres Bordeaux}, 8(2):331--346, 1996.

\bibitem{DajKra2002}
K.~Dajani and C.~Kraaikamp.
\newblock {\em Ergodic theory of numbers}, volume~29 of {\em Carus Mathematical
  Monographs}.
\newblock Mathematical Association of America, Washington, DC, 2002.

\bibitem{Fal2014}
Kenneth Falconer.
\newblock {\em Fractal geometry}.
\newblock John Wiley \& Sons, Ltd., Chichester, third edition, 2014.
\newblock Mathematical foundations and applications.

\bibitem{Gal1976}
J.~Galambos.
\newblock {\em Representations of real numbers by infinite series}.
\newblock Lecture Notes in Mathematics, Vol. 502. Springer-Verlag, Berlin-New
  York, 1976.

\bibitem{GeroGood}
G.~Gonz\'{a}lez~Robert.
\newblock Good's theorem for {H}urwitz continued fractions.
\newblock {\em Int. J. Number Theory}, 16(7):1433--1447, 2020.

\bibitem{Goo1941}
I.~J. Good.
\newblock The fractional dimensional theory of continued fractions.
\newblock {\em Proc. Cambridge Philos. Soc.}, 37:199--228, 1941.

\bibitem{HuaYe2002}
W.~Huang and X.~Ye.
\newblock Devaney's chaos or 2-scattering implies {L}i-{Y}orke's chaos.
\newblock {\em Topology Appl.}, 117(3):259--272, 2002.

\bibitem{JagDev1969}
H.~Jager and C.~de~Vroedt.
\newblock L\"{u}roth series and their ergodic properties.
\newblock {\em Nederl. Akad. Wetensch. Proc. Ser. A 72=Indag. Math.},
  31:31--42, 1969.

\bibitem{Jar1928}
V.~Jarník.
\newblock Zur metrischen theorie der diophantischen approximationen.
\newblock {\em Prace Matematyczno-Fizyczne}, 36(1):91--106, 1928-1929.

\bibitem{Khi1964}
A.~Y. Khinchin.
\newblock {\em Continued fractions}.
\newblock The University of Chicago Press, Chicago, Ill.-London, 1964.

\bibitem{KleWei2010}
D.~Kleinbock and B.~Weiss.
\newblock Modified {S}chmidt games and {D}iophantine approximation with
  weights.
\newblock {\em Adv. Math.}, 223(4):1276--1298, 2010.

\bibitem{LiYe2016}
J.~Li and X.~D. Ye.
\newblock Recent development of chaos theory in topological dynamics.
\newblock {\em Acta Math. Sin. (Engl. Ser.)}, 32(1):83--114, 2016.

\bibitem{LiYor1975}
T.~Y. Li and J.~A. Yorke.
\newblock Period three implies chaos.
\newblock {\em Amer. Math. Monthly}, 82(10):985--992, 1975.

\bibitem{LiuLi2017}
W.~Liu and B.~Li.
\newblock Chaotic and topological properties of continued fractions.
\newblock {\em J. Number Theory}, 174:369--383, 2017.

\bibitem{LiuWan2019}
W.~Liu and S.~Wang.
\newblock {The distal and asymptotic sets for continued fractions}.
\newblock {\em {Fractals}}, 27(8):10, 2019.
\newblock Id/No 1950139.

\bibitem{Lur1883}
J.~L\"{u}roth.
\newblock Ueber eine eindeutige {E}ntwickelung von {Z}ahlen in eine unendliche
  {R}eihe.
\newblock {\em Math. Ann.}, 21(3):411--423, 1883.

\bibitem{Neu2010}
J.~Neunh\"{a}userer.
\newblock Li-{Y}orke pairs of full {H}ausdorff dimension for some chaotic
  dynamical systems.
\newblock {\em Math. Bohem.}, 135(3):279--289, 2010.

\bibitem{Phi1967}
W.~Philipp.
\newblock Some metrical theorems in number theory.
\newblock {\em Pacific J. Math.}, 20:109--127, 1967.

\bibitem{Rue2017}
S.~Ruette.
\newblock {\em Chaos on the interval}, volume~67 of {\em University Lecture
  Series}.
\newblock American Mathematical Society, Providence, RI, 2017.

\bibitem{TanZho2021}
B.~Tan and Q.~Zhou.
\newblock Approximation properties of {L}\"{u}roth expansions.
\newblock {\em Discrete Contin. Dyn. Syst.}, 41(6):2873--2890, 2021.

\bibitem{TanZha2020}
X.~Tan and Z.~Zhang.
\newblock The relative growth rate for the digits in {L}\"{u}roth expansions.
\newblock {\em C. R. Math. Acad. Sci. Paris}, 358(5):557--562, 2020.

\bibitem{ViaOli2016}
M.~Viana and K.~Oliveira.
\newblock {\em Foundations of ergodic theory}, volume 151 of {\em Cambridge
  Studies in Advanced Mathematics}.
\newblock Cambridge University Press, Cambridge, 2016.

\bibitem{WanWu2008}
B.~W. Wang and J.~Wu.
\newblock Hausdorff dimension of certain sets arising in continued fraction
  expansions.
\newblock {\em Adv. Math.}, 218(5):1319--1339, 2008.

\end{thebibliography}
\bibliographystyle{plain}
\end{document}